\documentclass[12pt]{article}
\usepackage{bbm}
\usepackage{amsfonts}
\usepackage{pifont}
\usepackage{hyperref}
\usepackage{mathrsfs}
\usepackage{indentfirst}
\usepackage{amsmath}
\usepackage{amssymb}
\usepackage[amsmath, thmmarks]{ntheorem}
\usepackage{cite}
\usepackage{graphicx}
\usepackage{tikz}
\newtheorem{definition}{\bf Definition}[section]
\newtheorem{lemma}{\bf Lemma}[section]
\newtheorem{theorem}{\bf Theorem}[section]
\newtheorem{remark}{\bf Remark}[section]
\newtheorem{corollary}{\bf Corollary}[section]
\newtheorem{proposition}{\bf Proposition}[section]
\newtheorem{example}{\bf Example}[section]

\begin{document}
\setcounter{page}{1}

\title{{\textbf{
Some further construction methods  for  uninorms  via  uninorms
}}\thanks {Supported by National Natural Science
Foundation of China (No.11871097, 12271036)}}
\author{Zhenyu Xiu$^1$\footnote{\emph{E-mail address}: xzy@cuit.edu.cn }, Xu Zheng$^2$\footnote{Corresponding author. \emph{E-mail address}: 3026217474@qq.com }\\
\emph{\small $^{1,2}$  College of Applied Mathematics, Chengdu University of Information Technology, }\\
\emph{\small Chengdu 610000,  China }}
\newcommand{\pp}[2]{\frac{\partial #1}{\partial #2}}
\date{}
\maketitle

\begin{quote}
{\bf Abstract}

In this paper,  we further investigate new construction methods for uninorms on  bounded lattices  via  given uninorms.
  More specifically,   we first construct new uninorms on  arbitrary  bounded lattices   by extending a given uninorm  on a subinterval of the lattices under necessary and sufficient conditions on the given uninorm.  Moreover, based on the new uninorms,  we can give another  sufficient and necessary condition under which $S_{1}^{*}$ is a $t$-conorm ($T_{1}^{*}$ is a $t$-norm) in \cite{SS06}. Furthermore,  using  closure operators (interior operators), we also provide new construction methods for uninorms  by extending the given uninorm  on a subinterval of  a bounded lattice  under some additional constraints and  simultaneously  investigate the additional constraints carefully and systematically.
  Meanwhile,   some illustrative examples for the above construction methods of uninorms on bounded lattices are provided.

\end{quote}

\section{Introduction}\label{intro}

Yager and Rybalor \cite{RR96} introduced the notions of uninorms with a neutral element in the interior of the unit interval $[0,1]$ which are generalizations of $t$-norms and $t$-conorms. These operators also have been proved to paly an important role in other fields, such as neural networks, decision-making, expert systems and so on (see, e.g., \cite{BD99,MG09,MG11,MG15,PH85,PH92,WP07,EH76,RR01,RR03}).


Since the bounded lattice  $L$  is more general than  $[0,1]$, the studies of uninorms on $[0,1]$  have been extended to  $L$ .
Uninorms on $L$, were first  proposed in \cite{FK15}
as a unification of $t$-norms and $t$-conorms on $L$.
Since then, a lot of researchers have used many tools to construct  uninorms on the bounded lattices, such as $t$-norms ($t$-conorms) (see, e.g., \cite{EA21,EA22,SB14,GD16,GD17,GD018b,GD19,GD19.1,GD19.2,GD20,YD19,YD20,FK15,AX20}), closure operators (interior operators) (see, e.g., \cite{GD21,GD23,XJ21,YO20,BZ21}), $t$-subnorms ($t$-subconorms) (see, e.g., \cite{XJ20, WJ21, ZY23, HP21}),  additive generators \cite{HeP} and uninorms (see, e.g., \cite{GD22,ZX23}).

 More specifically, in \cite{GD22} and \cite{ZX23},
 the  researchers introduced  new approaches to  construct uninorms on  $L$  via given uninorms defined on a subinterval of  $L$, respectively.
In \cite{GD22},  G.D. \c{C}ayl{\i}  et al.  obtained new uninorms based on the presence of  $t$-norms ($t$-conorms) and   uninorms  defined on a  subinterval of  $L$ under some  additional constraints on $L$.
In \cite{ZX23}, Xiu and Zheng  gave new methods to yield uninorms on bounded lattices using  the presence of   $t$-superconorms ($t$-subnorms) and uninorms under some  additional constraints.  As we see, in fact, the above methods to construct uninorms both started from a given  uninorm  on a subinterval of  a bounded lattice  and  then  provided  a novel perspective to study the constructions of  uninorms.

In this paper, we still study the construction methods for uninorms via   uninorms defined on the subinterval $[0,a]$ (or $[b,1]$) of $L$.
 In section 3,  we can construct new uninorms by extending a given uninorm $U^{*}$ on a
subinterval  of $L$ under necessary and sufficient conditions on $U^{*}$.
The resulting uninorms can   provide  another  sufficient and necessary condition under which $S_{1}^{*}$ is a $t$-conorm ($T_{1}^{*}$ is a $t$-norm) in \cite{SS06}.
 In section 4,  based on  closure operators (interior operators),  we can  obtain new uninorms  by extending a given uninorm $U^{*}$ on a
subinterval of $L$.  Moreover, In case of   $e=b$,   $e=a$ or $e=0$,  all uninorms, constructed in sections 3 and  4, are the existed results in the literature.
  Meanwhile,  we  provide some illustrative examples for the above construction methods of uninorms.

\section{Preliminaries}

In this section, we recall some conceptions and  results,  which will be used in this manuscript.


\begin{definition}[\cite{GB67}]\label{de2.1}
A lattice $(L,\leq)$  is bounded if it has top element 1 and bottom  element 0, that is, there exist $1,0\in L$  such that $0\leq x\leq 1$ for all $x\in L$.
\end{definition}

Throughout this article, unless stated otherwise, we denote $L$  as a bounded lattice in the above definition.

\begin{definition}[\cite{GB67}]\label{de2.2}
Let  $a,b\in L$ with $a\leq b$. A subinterval $[a,b]$ of $L$ is defined as
$$[a, b]=\{x\in L: a\leq x \leq b\}.$$
Similarly, we can define $[a, b)=\{x\in L: a\leq x < b\}, (a, b]=\{x\in L: a< x \leq b\}$ and $(a, b)=\{x\in L: a< x < b\}$. If $a$  and $b$  are incomparable, then we use the notation $a\parallel b$. If $a$  and $b$  are comparable, then we use the notation $a\nparallel b$.
\end{definition}

Moreover, $I_{a}$ denotes the set of all incomparable elements with $a$, that is, $I_{a}=\{x\in L\mid x\parallel a\}$. $I_{a}^{b}$ denotes the set of elements that are incomparable with $a$ but comparable with $b$, that is, $I_{a}^{b}=\{x\in L\mid x \parallel a\ and\ x\nparallel b\}$.   $I_{a,b}$ denotes the set of elements that are incomparable with both $a$ and $b$, that is, $I_{a,b}=\{x\in L \mid x \parallel a\ and \ x \parallel b\}$. Obviously, $I_{a}^{a}=\emptyset$ and $I_{a,a}=I_{a}$.

\begin{definition}[\cite{SS06}]\label{de2.3}
If  an operation $T:L^{2}\rightarrow L$  is associative, commutative and increasing with respect to both variables, and   has the neutral element $1\in L$, that is, $T(1,x)=x$   for all $x\in L$, then it is called a $t$-norm on $L$.

\end{definition}

\begin{definition}[\cite{GD16}]\label{}
If an operation $S:L^{2}\rightarrow L$ is associative, commutative and increasing with respect to both variables, and   has the neutral element $0\in L$, that is, $S(0,x)=x$  for all $x\in L$,  then it is called a $t$-conorm on  $L$.
\end{definition}

\begin{definition}[\cite{FK15}]\label{de2.4}
If an operation $U:L^{2}\rightarrow L$ is  associative, commutative and increasing with respect to both variables, and   has the neutral element $e\in L$, that is, $U(e,x)=x$  for all $x\in L$,  then it  is called a uninorm on  $L$.
\end{definition}

\begin{definition}[\cite{FK15}]\label{de34}
Let $U$ be a uninorm on $L$ with the neutral element $e\in L\setminus\{0,1\}$. Then the following results hold obviously.\\
$(1)$ $T_{e}=U\mid [0,e]^{2}\rightarrow [0,e]$ is a $t$-norm on $[0,e]$.\\
$(2)$ $S_{e}=U\mid [e,1]^{2}\rightarrow [e,1]$ is a $t$-conorm on $[e,1]$.
\end{definition}

\begin{definition}[\cite{GD16}]\label{de32}
Let $U$ be a uninorm on $L$ with the neutral element $e\in L\setminus\{0,1\}$.\\
$(1)$ If  an element $x\in L$ satisfies $U(x,x)=x$, then it is called an idempotent element of $U$.\\
$(2)$ If a uninorm $U$  satisfies  $U(x,x)=x$ for all $x\in L$, then it   is called an idempotent uninorm on $L$.
\end{definition}

\begin{definition}[\cite{GD16}]\label{de33}
Let $U$ be a uninorm on $L$ with the neutral element $e\in L\setminus\{0,1\}$.\\
$(1)$ If $U(0,1)=0$, then $U$ is a conjunctive uninorm.\\
$(2)$ If $U(0,1)=1$, then $U$ is a disjunctive uninorm.
\end{definition}

\begin{definition}[\cite{CJ44}]
A mapping $cl: L^{2}\rightarrow L$ is called  a closure operator on $L$ if, for all $x,y\in L$, it satisfies the following conditions:\\
$(1)$ $x\leq cl(x)$;\\
$(2)$ $cl(x\vee y)= cl(x)\vee cl(y)$;\\
$(3)$ $cl(cl(x))= cl(x)$.
\end{definition}

\begin{definition}[\cite{YO20}]
A mapping $int: L^{2}\rightarrow L$ is called  an interior operator on $L$ if, for all $x,y\in L$, it satisfies the following  conditions:\\
$(1)$ $int(x)\leq x$;\\
$(2)$ $int(x\wedge y)= int(x)\wedge int(y)$;\\
$(3)$ $int(int(x))= int(x)$.
\end{definition}

\begin{definition}[\cite{HP21}]\label{de35}
In \cite{HP21},  two classes of uninorms  $\mathcal{U}_{min}^{*}$ and  $\mathcal{U}_{max}^{*}$  on  bounded lattices were  introduced  by  H.-P. Zhang et al.,  where
 $\mathcal{U}_{min}^{*}$  denotes  the class of all uninorms $U$ on $L$ with neutral element $e$ satisfying  $U(x,y)=y$ for all $(x,y)\in(e,1]\times [0,e)$ and    $\mathcal{U}_{max}^{*}$ denotes   the class of all uninorms $U$ on $L$ with neutral element $e$ satisfying
$U(x,y)=y$ for all $(x,y)\in[0,e)\times (e,1]$.
\end{definition}

\begin{proposition}[\cite{WJ21}]\label{pro2.1}
Let $S$ be a nonempty set and $A_{1},A_{2},\ldots,A_{n}$ be subsets of $S$. Let $H$ be a commutative binary operation on $S$. Then $H$ is associative on $A_{1}\cup A_{2}\cup\ldots\cup A_{n}$ if and only if all of the following statements hold:\\
$\mathrm{(i)}$ for every combination $\{i,j,k\}$ of size $3$ chosen from $\{1,2,\ldots,n\}$, $H(x,H(y,z))=H(H(x,y),z)=H(y,H(x,z))$ for all $x\in A_{i},y\in A_{j},z\in A_{k}$;\\
$(ii)$ for every combination $\{i,j\}$ of size $2$ chosen from $\{1,2,\ldots,n\}$, $H(x,H(y,z))=H(H(x,y),z) $ for all $x\in A_{i},y\in A_{i},z\in A_{j}$;\\
$(iii)$ for every combination $\{i,j\}$ of size $2$ chosen from $\{1,2,\ldots,n\}$, $H(x,H(y,z))=H(H(x,y),z) $ for all $x\in A_{i},y\in A_{j},z\in A_{j}$;\\
$(iv)$ for every $i\in \{1,2,\ldots,n\}$, $H(x,H(y,z))=H(H(x,y),z) $ for all $x,y,z\in A_{i}$.
\end{proposition}

\begin{theorem}[\cite{SS06}]\label{th250}
Let $b, a\in L\setminus \{0,1\}$.\\
$(1)$ For a
$t$-norm $V:[b, 1]^{2}\rightarrow[b, 1]$,  an ordinal sum
extension $T_{1}^{*}$ of $V$ to $L$ defined by

$T_{1}^{*}(x,y)=\begin{cases}
V(x, y) &\mbox{if } (x,y)\in [b,1]^{2},\\
x\wedge y &\mbox{}otherwise,\\
\end{cases}$\\
is a $t$-norm if and only if $x\parallel y$ for all $x\in I_{b}$ and $y\in [b,1)$.\\
$(2)$ For  a
$t$-conorm $W:[0,a]^{2}\rightarrow[0,a]$, an ordinal sum
extension $S_{1}^{*}$ of $W$ to $L$ defined by

$S_{1}^{*}(x,y)=\begin{cases}
W(x, y) &\mbox{if } (x,y)\in [0,a]^{2},\\
x\vee y &\mbox{}otherwise,\\
\end{cases}$\\
is a $t$-conorm if and only if $x\parallel y$ for all $x\in I_{a}$ and $y\in (0,a]$.
\end{theorem}

\section{New construction methods for uninorms via uninorms on bounded lattices}

In this section, we focus on the construction methods for uninorms on a  bounded lattice $L$  by extending a given uninorm $U^{*}$ on a subinterval $[0,a]$ (or $[b,1]$ )  of  $L$   under sufficient and necessary conditions on $U^{*}$.
Based on the new uninorms, we can obtain  another  sufficient and necessary condition under which $S_{1}^{*}$ is a $t$-conorm ($T_{1}^{*}$ is a $t$-norm) in \cite{SS06}.

\begin{theorem}\label{th31}
Let  $U^{*}$ be a uninorm on $[0,a]$ with a neutral element $e\in L$ for   $a\in L\setminus\{0,1\}$. Then $U_{1}(x,y):L^{2}\rightarrow L$ defined by

$U_{1}(x,y)=\begin{cases}
U^{*}(x,y) &\mbox{if } (x,y)\in [0,a]^{2},\\
x &\mbox{if } (x,y)\in I_{e,a}\times [0,e],\\
y &\mbox{if } (x,y)\in [0,e]\times I_{e,a},\\
x\vee y &\mbox{}otherwise,\\
\end{cases}$\\
is a uninorm on $L$ with the neutral element $e$ if and only if $U^{*}$ satisfies  the following conditions: \\
$(1)$ for   $z\in I_{e,a}$, if   $(x,y)\in ((I_{e}^{a}\cup (e,a])\cap \{l\in L\ | \ l \nparallel z\})^{2}$, then  $U^{*}(x, y)\nparallel z  $;\\
$(2)$ for   $z\in I_{e,a}$,  if  $(x,y)\in[0,a]\times((I_{e}^{a}\cup (e,a])\cap\{l\in L\ |\ l \parallel z\})\cup((I_{e}^{a}\cup (e,a])\cap\{l\in L\ |\ l \parallel z\})\times[0,a]$, then $U^{*}(x, y)\parallel z $;\\
$(3)$ for   $z\in I_{a}^{e}$, if   $(x,y)\in ((I_{e}^{a}\cup (e,a])\cap \{l\in L\ | \ l \nparallel z\})^{2}$, then  $U^{*}(x, y)\nparallel z $;\\
$(4)$ for   $z\in I_{a}^{e}$,  if  $(x,y)\in[0,a]\times((I_{e}^{a}\cup (e,a])\cap\{l\in L\ |\ l \parallel z\})\cup((I_{e}^{a}\cup (e,a])\cap\{l\in L\ |\ l \parallel z\})\times[0,a]$, then $U^{*}(x, y)\parallel z  $.
\end{theorem}

\begin{proof}
Necessity. Suppose that  $U_{1}(x,y)$ is a uninorm with the neutral element $e$.  Next we need to show  that $U^{*}$ satisfies the conditions (1), (2), (3) and (4).

$(1)$.  For   $z\in I_{e,a}$, if  $(x,y)\in ((I_{e}^{a}\cup (e,a])\cap \{l\in L\ | \ l \nparallel z\})^{2}$,  then $U^{*}(x,y)\nparallel z$.

Assume that for  $z\in I_{e,a}$,  there exists $(x,y)\in ((I_{e}^{a}\cup (e,a])\cap \{l\in L\ | \ l \nparallel z\})^{2}$ such that  $U^{*}(x,y)\parallel z$. Then $U_{1}(x,z)=z$ and $U_{1}(x,y)=U^{*}(x,y)$. Since $U^{*}(x,y)\parallel z$, this  contradicts the increasingness property of $U_{1}(x,y)$. Thus, for   $z\in I_{e,a}$, if  $(x,y)\in ((I_{e}^{a}\cup (e,a])\cap \{l\in L\ | \ l \nparallel z\})^{2}$, then $U^{*}(x,y)\nparallel z$.

$(2)$.  For  $z\in I_{e,a}$, if   $(x,y)\in[0,a]\times((I_{e}^{a}\cup (e,a])\cap\{l\in L\ |\ l \parallel z\})\cup((I_{e}^{a}\cup (e,a])\cap\{l\in L\ |\ l \parallel z\})\times[0,a]$, then $U^{*}(x,y)\parallel z$.

Now we give the proof of that for  $z\in I_{e,a}$, if   $(x,y)\in[0,a]\times((I_{e}^{a}\cup(e,a])\cap\{l\in L\ |\ l \parallel z\})$, then $U^{*}(x,y)\parallel z$, and the other case is obvious by the commutativity  of $U^{*}$. Assume that for $z\in I_{e,a}$, there exists $(x,y)\in[0,a]\times((I_{e}^{a}\cup (e,a])\cap\{l\in L\ |\ l \parallel z\})$ such that  $U^{*}(x,y)\nparallel z$. Then $U_{1}(x,U_{1}(y,z))=U_{1}(x,z\vee a)=z\vee a$ and $U_{1}(U_{1}(x,y),z)=U_{1}(U^{*}(x,y),z)=z$. Since $z\vee a\neq z$, this   contradicts  the associativity  of $U_{1}(x,y)$. Thus, for  $z\in I_{e,a}$, if  $(x,y)\in[0,a]\times((I_{e}^{a}\cup(e,a])\cap\{l\in L\ |\ l \parallel z\})\cup((I_{e}^{a}\cup(e,a])\cap\{l\in L\ |\ l \parallel z\})\times[0,a]$, then $U^{*}(x,y)\parallel z$.

$(3)$.  For  $z\in I_{a}^{e}$, if $(x,y)\in((I_{e}^{a}\cup (e,a])\cap\{l\in L\ |\ l \nparallel z\})^{2}$, then $U^{*}(x,y)\nparallel z  $.

Assume that for $z\in I_{a}^{e}$, there exists $(x,y)\in((I_{e}^{a}\cup (e,a])\cap\{l\in L\ |\ l \nparallel z\})^{2} $ such that $U^{*}(x,y)\parallel z  $. Then $U_{1}(x,U_{1}(y,z))=U_{1}(x,z)=z$ and $U_{1}(U_{1}(x,y),z)=U_{1}(U^{*}(x,y),z)=z\vee a$. Since $z\vee a\neq z$, this   contradicts the associativity   of $U_{1}$. Thus, for   $z\in I_{a}^{e}$, if  $(x,y)\in((I_{e}^{a}\cup (e,a])\cap\{l\in L\ |\ l \nparallel z\})^{2}$, then $U^{*}(x,y)\nparallel z  $.

$(4)$. For $z\in I_{a}^{e}$, if $(x,y)\in[0,a]\times((I_{e}^{a}\cup(e,a])\cap\{l\in L\ |\ l \parallel a\})\cup ((I_{e}^{a}\cup(e,a])\cap\{l\in L\ |\ l \parallel z\})\times[0,a]$, then $U^{*}(x,y)\parallel z  $.

Now we prove that for $z\in I_{a}^{e}$, if $(x,y)\in[0,a]\times((I_{e}^{a}\cup(e,a])\cap\{l\in L\ |\ l \parallel z\})$, then $U^{*}(x,y)\parallel z  $, and the other case is obvious   by the commutativity  of $U^{*}$. Assume that for $z\in I_{a}^{e}$, there exists $(x,y)\in[0,a]\times((I_{e}^{a}\cup(e,a])\cap\{l\in L\ |\ l \parallel z\})$ such that   $U^{*}(x,y)\nparallel z  $. Then $U_{1}(x,U_{1}(y,z))=U_{1}(x,z\vee a)=z\vee a$ and $U_{1}(U_{1}(x,y),z)=U_{1}(U^{*}(x,y),z)=z$. Since $z\vee a\neq z$, this  contradicts the associativity  of $U_{1}$. Thus, for $z\in I_{a}^{e}$, if $(x,y)\in[0,a]\times((I_{e}^{a}\cup(e,a])\cap\{l\in L\ |\ l \parallel a\})\cup ((I_{e}^{a}\cup(e,a])\cap\{l\in L\ |\ l \parallel z\})\times[0,a]$, then $U^{*}(x,y)\parallel z  $.

Sufficiency.  First, we can see that $U_{1}$ is commutative and  $e$  is the neutral element of $U_{1}$.  Hence, we only need to prove the increasingness and the associativity of $U_{1}$.

I. Increasingness: Next, we prove that if $x\leq y$, then $U_{1}(x,z)\leq U_{1}(y,z)$ for all $z\in L$. It is easy to verify that $U_{1}(x,z)\leq U_{1}(y,z)$ if both $x$ and $y$ belong to one of the  intervals $ [0,e], I_{e}^{a}, (e,a], I_{a}^{e}, I_{e,a}$ or $(a,1]$ for all $z\in L$. The residual proof can be split into all possible cases.

1. $x\in [0,e]$

\ \ \ 1.1. $y\in I_{e}^{a}\cup (e,a]$

\ \ \ \ \ \ 1.1.1. $z\in [0,e]\cup I_{e}^{a}\cup (e,a]$

\ \ \ \ \ \ \ \ \ \ \ \ $U_{1}(x,z)=U^{*}(x,z)\leq U^{*}(y,z)=U_{1}(y,z)$

\ \ \ \ \ \ 1.1.2. $z\in I_{a}^{e}\cup I_{e,a}\cup (a,1]$

\ \ \ \ \ \ \ \ \ \ \ \ $U_{1}(x,z)=z\leq y\vee z =U_{1}(y,z)$

\ \ \ 1.2. $y\in I_{a}^{e}\cup I_{e,a}\cup (a,1]$

\ \ \ \ \ \ 1.2.1. $z\in [0,e]$

\ \ \ \ \ \ \ \ \ \ \ \ $U_{1}(x,z)=U^{*}(x,z)\leq x<y =U_{1}(y,z)$

\ \ \ \ \ \ 1.2.2. $z\in  I_{e}^{a}\cup (e,a]$

\ \ \ \ \ \ \ \ \ \ \ \ $U_{1}(x,z)=U^{*}(x,z)\leq z<y\vee z =U_{1}(y,z)$

\ \ \ \ \ \ 1.2.3. $z\in I_{a}^{e}\cup I_{e,a}\cup (a,1]$

\ \ \ \ \ \ \ \ \ \ \ \ $U_{1}(x,z)=z\leq y\vee z =U_{1}(y,z)$

2. $x\in I_{e}^{a}$

\ \ \ 2.1. $y\in (e,a]$

\ \ \ \ \ \ 2.1.1. $z\in [0,e]\cup I_{e}^{a}\cup (e,a]$

\ \ \ \ \ \ \ \ \ \ \ \ $U_{1}(x,z)=U^{*}(x,z)\leq U^{*}(y,z)=U_{1}(y,z)$

\ \ \ \ \ \ 2.1.2. $z\in I_{a}^{e}\cup I_{e,a}\cup (a,1]$

\ \ \ \ \ \ \ \ \ \ \ \ $U_{1}(x,z)=x\vee z\leq y\vee z =U_{1}(y,z)$

\ \ \ 2.2. $y\in I_{a}^{e} $

\ \ \ \ \ \ 2.2.1. $z\in [0,e]$

\ \ \ \ \ \ \ \ \ \ \ \ $U_{1}(x,z)=U^{*}(x,z)\leq x<y =U_{1}(y,z)$

\ \ \ \ \ \ 2.2.2. $z\in I_{e}^{a}\cup (e,a]$

\ \ \ \ \ \ \ \ \ \ \ \ If $z\parallel y$, then $U_{1}(x,z)=U^{*}(x,z)\leq a<y\vee a= y\vee z =U_{1}(y,z)$.

\ \ \ \ \ \ \ \ \ \ \ \ If $z\nparallel y$, then $U_{1}(x,z)=U^{*}(x,z) < y =y\vee z =U_{1}(y,z)$.

\ \ \ \ \ \ 2.2.3. $z\in I_{a}^{e}\cup I_{e,a} \cup(a,1]$

\ \ \ \ \ \ \ \ \ \ \ \ $U_{1}(x,z)=x\vee z\leq y\vee z =U_{1}(y,z)$

\ \ \ 2.3. $y\in I_{e,a}$

\ \ \ \ \ \ 2.3.1. $z\in [0,e]$

\ \ \ \ \ \ \ \ \ \ \ \ $U_{1}(x,z)=U^{*}(x,z)\leq x<y =U_{1}(y,z)$

\ \ \ \ \ \ 2.3.2. $z\in  I_{e}^{a}$

\ \ \ \ \ \ \ \ \ \ \ \ If $z\parallel y$, then $U_{1}(x,z)=U^{*}(x,z)\leq a<y\vee a= y\vee z =U_{1}(y,z)$.

\ \ \ \ \ \ \ \ \ \ \ \ If $z\nparallel y$, then $U_{1}(x,z)=U^{*}(x,z) < y =y\vee z =U_{1}(y,z)$.

\ \ \ \ \ \ 2.3.3. $z\in   (e,a]$

\ \ \ \ \ \ \ \ \ \ \ \ $U_{1}(x,z)=U^{*}(x,z)\leq a<y\vee a=y\vee z =U_{1}(y,z)$

\ \ \ \ \ \ 2.3.4. $z\in I_{a}^{e}\cup I_{e,a}\cup (a,1]$

\ \ \ \ \ \ \ \ \ \ \ \ $U_{1}(x,z)=x\vee z\leq y\vee z =U_{1}(y,z)$

\ \ \ 2.4. $y\in  (a,1]$

\ \ \ \ \ \ 2.4.1. $z\in [0,e]\cup I_{e}^{a}\cup (e,a]$

\ \ \ \ \ \ \ \ \ \ \ \ $U_{1}(x,z)=U^{*}(x,z)\leq a<y =U_{1}(y,z)$

\ \ \ \ \ \ 2.4.2. $z\in I_{a}^{e}\cup I_{e,a} \cup(a,1]$

\ \ \ \ \ \ \ \ \ \ \ \ $U_{1}(x,z)=x\vee z\leq y\vee z =U_{1}(y,z)$

3. $x\in (e,a]$

\ \ \ 3.1. $y\in I_{a}^{e}$

\ \ \ \ \ \ 3.1.1. $z\in [0,e]$

\ \ \ \ \ \ \ \ \ \ \ \ $U_{1}(x,z)=U^{*}(x,z)\leq x<y =U_{1}(y,z)$

\ \ \ \ \ \ 3.1.2. $z\in I_{e}^{a}\cup (e,a]$

\ \ \ \ \ \ \ \ \ \ \ \ If $z\parallel y$, then $U_{1}(x,z)=U^{*}(x,z)\leq a<y\vee a= y\vee z =U_{1}(y,z)$.

\ \ \ \ \ \ \ \ \ \ \ \ If $z\nparallel y$, then $U_{1}(x,z)=U^{*}(x,z) < y =y\vee z =U_{1}(y,z)$.

\ \ \ \ \ \ 3.1.3. $z\in I_{a}^{e}\cup I_{e,a}\cup (a,1]$

\ \ \ \ \ \ \ \ \ \ \ \ $U_{1}(x,z)=x\vee z\leq y\vee z =U_{1}(y,z)$

\ \ \ 3.2. $y\in (a,1]$

\ \ \ \ \ \ 3.2.1. $z\in [0,e]\cup I_{e}^{a}\cup (e,a]$

\ \ \ \ \ \ \ \ \ \ \ \ $U_{1}(x,z)=U^{*}(x,z)\leq a<y =U_{1}(y,z)$

\ \ \ \ \ \ 3.2.2. $z\in I_{a}^{e}\cup I_{e,a}\cup (a,1]$

\ \ \ \ \ \ \ \ \ \ \ \ $U_{1}(x,z)=x\vee z\leq y\vee z =U_{1}(y,z)$

4. $x\in I_{a}^{e}, y\in (a,1], z\in L$

\ \ \ \ \ \ \ \ \ \ \ \ $U_{1}(x,z)=x\vee z\leq y\vee z =U_{1}(y,z)$

5. $x\in I_{e,a}, y\in I_{a}^{e}\cup(a,1]$

\ \ \ 5.1. $z\in [0,e]$

\ \ \ \ \ \ \ \ \ \ \ \ $U_{1}(x,z)=x\leq y =U_{1}(y,z)$

\ \ \ 5.2. $z\in I_{e}^{a}\cup(e,a]\cup I_{a}^{e}\cup I_{e,a}\cup(a,1]$

\ \ \ \ \ \ \ \ \ \ \ \ $U_{1}(x,z)=x\vee z\leq y\vee z =U_{1}(y,z)$

II. Associativity: It can be shown that $U_{1}(x,U_{1}(y,z))$ $=U_{1}(U_{1}(x,y),z)$  for all $x,y,z\in L$. By Proposition \ref{pro2.1}, we just  verify the following cases.

1. If $x,y,z\in [0,e]\cup I_{e}^{a}\cup  (e,a]$, then $U_{1}(x, U_{1}(y,z))=U_{1}(U_{1}(x,y),z)=U_{1}(y, U_{1}(x,z))$ for   $U^{*}$ is associative.

2. If $x,y,z\in I_{a}^{e}\cup I_{e,a}\cup (a,1] $, then $U_{1}(x, U_{1}(y,z))=U_{1}(x,y\vee z)=x\vee y\vee z =U_{1}(x\vee y,z)=U_{1}(U_{1}(x,y),z) $ and $U_{1}(y, U_{1}(x,z))=U_{1}(y,x\vee z)=x\vee y\vee z$. Thus $U_{1}(x, U_{1}(y,z))=U_{1}(U_{1}(x,y),z)=U_{1}(y, U_{1}(x,z))$.

3. If $x,y\in [0,e]$ and $z\in I_{a}^{e}\cup I_{e,a}\cup (a,1] $, then $U_{1}(x, U_{1}(y,z))=U_{1}(x,z)=z =U_{1}(U^{*}(x,y),z)=U_{1}(U_{1}(x,y),z)$.

4. If $x,y\in I_{e}^{a}, z\in I_{a}^{e}\cup I_{e,a}, x\nparallel z$ and $y\nparallel z$, then $U_{1}(x, U_{1}(y,z))=U_{1}(x,z)=z =U_{1}(U^{*}(x,y),z)=U_{1}(U_{1}(x,y),z)$.

\ If $x,y\in I_{e}^{a}, z\in I_{a}^{e}\cup I_{e,a} $ and at least one of $x,y$ is incomparable with $z$, then $U_{1}(x, U_{1}(y,z))=U_{1}(x,y\vee z)=x\vee y\vee z =z\vee a=U^{*}(x,y)\vee z =U_{1}(U^{*}(x,y),z)=U_{1}(U_{1}(x,y),z) $.

5. If $x,y\in I_{e}^{a}$ and $z\in (a,1]$, then $U_{1}(x, U_{1}(y,z))=U_{1}(x,z)=z =U_{1}(U^{*}(x,y),z)=U_{1}(U_{1}(x,y),z) $.

6. If $x,y\in (e,a], z\in I_{a}^{e} , x\nparallel z$ and $y\nparallel z$, then $U_{1}(x, U_{1}(y,z))=U_{1}(x,z)=z =U_{1}(U^{*}(x,y),z)=U_{1}(U_{1}(x,y),z)$.

\ If  $x,y\in (e,a], z\in I_{a}^{e}$ and at least one of $x$ and $y$ is incomparable with $z$,  then $U_{1}(x, U_{1}(y,z))=U_{1}(x,y\vee z)=x\vee y\vee z =z\vee a=U^{*}(x,y)\vee z =U_{1}(U^{*}(x,y),z)=U_{1}(U_{1}(x,y),z)$.

7. If $x,y\in (e,a]$ and $z\in I_{e,a}$, then $U_{1}(x, U_{1}(y,z))$ $=U_{1}(x,y\vee z)=x\vee y\vee z =z\vee a=U^{*}(x,y)\vee z =U_{1}(U^{*}(x,y),z)=U_{1}(U_{1}(x,y),z) $.

8. If $x,y\in (e,a]$ and $z\in (a,1]$, then $U_{1}(x, U_{1}(y,z))$ $ =U_{1}(x,z)=z =U_{1}(U^{*}(x,y),z)$ $=U_{1}(U_{1}(x,y),z) $.

9. If $x\in [0,e]$ and $y,z\in I_{a}^{e}\cup I_{e,a}\cup (a,1] $, then $U_{1}(x, U_{1}(y,z))=U_{1}(x,y \vee z)=y\vee z =U_{1}(y,z)=U_{1}(U_{1}(x,y),z) $ and $U_{1}(y, U_{1}(x,z))=U_{1}(y,z)=y\vee z$. Thus $U_{1}(x, U_{1}(y,z))=U_{1}(U_{1}(x,y),z)=U_{1}(y, U_{1}(x,z))$.

10. If $x\in I_{e}^{a}\cup (e,a]$, $y,z\in I_{a}^{e}\cup I_{e,a}\cup (a,1] $, then $U_{1}(x, U_{1}(y,z))=U_{1}(x,y\vee z)=x\vee y\vee z =U_{1}(x\vee y,z)=U_{1}(U_{1}(x,y),z) $ and $U_{1}(y, U_{1}(x,z))=U_{1}(y,x\vee z)=x\vee y\vee z$. Thus $U_{1}(x, U_{1}(y,z))=U_{1}(U_{1}(x,y),z)=U_{1}(y, U_{1}(x,z))$.

11. If $x\in [0,e], y\in I_{e}^{a}, z\in I_{a}^{e}\cup I_{e,a}$ and $y\nparallel z$, then $U_{1}(x, U_{1}(y,z))=U_{1}(x,z)=z =U_{1}(U^{*}(x,y),z)= U_{1}(U_{1}(x,y),z)  $ and $U_{1}(y, U_{1}(x,z))$ $=U_{1}(y,z)=z$. Thus $U_{1}(x, U_{1}(y,z))=U_{1}(U_{1}(x,y),$ $z)=U_{1}(y, U_{1}(x,z))$.

\ If $x\in [0,e], y\in I_{e}^{a}, z\in I_{a}^{e}\cup I_{e,a}$ and $y\parallel z$, then $U_{1}(x, U_{1}(y,z))=U_{1}(x,y\vee z)=x\vee y\vee z =z\vee a=U^{*}(x,y)\vee z =U_{1}(U^{*}(x,y),z)= U_{1}(U_{1}(x,y),z) $ and $U_{1}(y, U_{1}(x,z))=U_{1}(y,z)=y\vee z =z\vee a$. Thus $U_{1}(x, U_{1}(y,z))=U_{1}(U_{1}(x,y),z)=U_{1}(y, U_{1}(x,z))$.

12. If $x\in [0,e],y\in I_{e}^{a}$ and $z\in (a,1]$, then $U_{1}(x, U_{1}(y,z))=U_{1}(x,z)=z =U_{1}(U^{*}(x,y),z)$ $= U_{1}(U_{1}(x,y),z) $ and $U_{1}(y, U_{1}(x,z))=U_{1}(y,z)=z$. Thus $U_{1}(x, U_{1}(y,z))$ $=U_{1}(U_{1}(x,y),z)=U_{1}(y, U_{1}$ $(x,z))$.

13. If $x\in [0,e], y\in (e,a], z\in I_{a}^{e}$ and $y\nparallel z$, then $U_{1}(x, U_{1}(y,z))=U_{1}(x,z)=z =U_{1}(U^{*}(x,y),z)= U_{1}(U_{1}(x,y),z)  $ and $U_{1}(y, U_{1}(x,z))=U_{1}(y,z)=z$. Thus $U_{1}(x, U_{1}(y,z))$ $=U_{1}(U_{1}(x,y),z)=U_{1}(y, U_{1}$ $(x,z))$.

\ If $x\in [0,e], y\in (e,a], z\in I_{a}^{e}$ and $y\parallel z$, then $U_{1}(x, U_{1}(y,z))=U_{1}(x,y\vee z)=x\vee y\vee z =z\vee a=U^{*}(x,y)\vee z =U_{1}(U^{*}(x,y),z)= U_{1}(U_{1}(x,y),z) $ and $U_{1}(y, U_{1}(x,z))=U_{1}(y,z)=y\vee z =z\vee a$. Thus $U_{1}(x, U_{1}(y,z))=U_{1}(U_{1}(x,y),z)=U_{1}(y, U_{1}(x,z))$.

14. If $x\in [0,e],y\in (e,a]$ and $z\in I_{e,a}$, then $U_{1}(x, U_{1}(y,z))=U_{1}(x,y\vee z)=x\vee y\vee z =z\vee a=U^{*}(x,y)\vee z =U_{1}(U^{*}(x,y),z)= U_{1}(U_{1}(x,y),z) $ and $U_{1}(y, U_{1}(x,z))=U_{1}(y,z)=y\vee z =z\vee a$. Thus $U_{1}(x, U_{1}(y,z))=U_{1}(U_{1}(x,y),z)=U_{1}(y, U_{1}(x,z))$.

15. If $x\in [0,e],y\in (e,a]$ and $z\in (a,1]$, then $U_{1}(x, U_{1}(y,z))=U_{1}(x,z)=z =U_{1}(U^{*}(x,y),z)=U_{1}(U_{1}(x,y),z) $ and $U_{1}(y, U_{1}(x,z))=U_{1}(y,z)=z$. Thus $U_{1}(x, U_{1}(y,z))$ $=U_{1}(U_{1}(x,y),z)=U_{1}(y, U_{1}$ $(x,z))$.

16. If $x\in I_{e}^{a}, y\in (e,a], z\in I_{a}^{e} , x\nparallel z$ and $y\nparallel z$, then $U_{1}(x, U_{1}(y,z))=U_{1}(x,z)=z =U_{1}(U^{*}(x,y),$ $z)= U_{1}(U_{1}(x,y),z) $ and $U_{1}(y, U_{1}(x,z)) =U_{1}(y,z)=z$. Thus $U_{1}(x, U_{1}(y,z))=U_{1}(U_{1}(x,y),z)=U_{1}(y, U_{1}(x,z))$.

\ If $x\in I_{e}^{a}, y\in (e,a],  z\in I_{a}^{e} $ and at least one of $x,y$ is incomparable with $z$, then $U_{1}(x, U_{1}(y,z))=U_{1}(x,y\vee z)=x\vee y\vee z =z\vee a=U^{*}(x,y)\vee z =U_{1}(U^{*}(x,y),z)=U_{1}(U_{1}(x,y),z) $ and $U_{1}(y, U_{1}(x,z))$ $=U_{1}(y,x\vee z)=y\vee x\vee z =z\vee a$. Thus $U_{1}(x, U_{1}(y,z))=U_{1}(U_{1}(x,y),z)=U_{1}(y, U_{1}(x,z))$.

17. If $x\in I_{e}^{a},y\in (e,a]$ and $z\in I_{e,a}$, then $U_{1}(x, U_{1}(y,z))=U_{1}(x,y\vee z)=x\vee y\vee z =z\vee a=U^{*}(x,y)\vee z =U_{1}(U^{*}(x,y),z)= U_{1}(U_{1}(x,y),z) $ and $U_{1}(y, U_{1}(x,z))=U_{1}(y,x\vee z)=x\vee y\vee z =z\vee a$. Thus $U_{1}(x, U_{1}(y,z))=U_{1}(U_{1}(x,y),z)=U_{1}(y, U_{1}(x,z))$.

18. If $x\in I_{e}^{a},y\in (e,a]$ and $z\in (a,1]$, then $U_{1}(x, U_{1}(y,z))=U_{1}(x,z)=z =U_{1}(U^{*}(x,y),z)$ $=U_{1}(U_{1}(x,y),z) $ and $U_{1}(y, U_{1}(x,z))=U_{1}(y,z)=z$. Thus $U_{1}(x, U_{1}(y,z))$ $=U_{1}(U_{1}(x,y),z)$ $=U_{1}(y, U_{1}$ $(x,z))$.
\end{proof}

\begin{remark}
Theorem 3.1  seem to be  restrained for there are some  conditions on the given uninorm $U^{*}$.  However, these conditions are necessary for our construction methods.
  On one hand, these additional conditions are necessary and sufficient;  on the other hand,  in case of $e=a$ or $e=0$, these conditions naturally hold and then  Theorem 3.1 is  the existed result in the literature as follows. These show the rationality of these conditions and our uninorms in some degree.
\end{remark}

\begin{remark}\label{re32}
 In Theorem \ref{th31}, if taking $e=a$, then  $[0,a]=[0,e]$, $I_{e,a}=I_{e}$, $I_{e}^{a}\cup I_{a}^{e}\cup (e,a]=\emptyset$ and $U^{*}$ is a $t$-norm on $[0,a]$. Moreover, based on the above case, the conditions $(1)$, $(2)$, $(3)$  and $(4)$ in Theorem \ref{th31} naturally hold.

By the above fact, if taking $e=a$ in Theorem \ref{th31}, then  we retrieve the uninorm $U_{t_{1}}: L^{2}\rightarrow L$ constructed by \c{C}ayl{\i}, Kara\c{c}al, and Mesiar (\cite{GD16}, Theorem 1) as follow.

$U_{t_{1}}(x,y)=\begin{cases}
T_{e}(x, y) &\mbox{if } (x,y)\in [0,e]^{2},\\
x &\mbox{if } (x,y)\in I_{e}\times[0,e],\\
y &\mbox{if } (x,y)\in [0,e]\times I_{e},\\
x\vee y &\mbox{}otherwise.\\
\end{cases}$
\end{remark}


\begin{lemma}\label{Le32}
In Theorem \ref{th31},  if  $e=0$,  then  the condition $(4)$  holds.
\end{lemma}
\begin{proof}
 If  $e=0$,  then we can  rewrite the condition $(4)$ as follow: for $z\in I_{a}$, if  $(x,y)\in[0,a]\times( (0,a] \cap\{l\in L\ |\ l \parallel z\})\cup((0,a] \cap\{l\in L\ |\ l \parallel z\})\times[0,a]$, then $U^{*}(x, y)\parallel z$.
Next we just prove that  for $z\in I_{a} $, if  $(x,y)\in[0,a]\times( (0,a] \cap\{l\in L\ |\ l \parallel z\})$, then $U^{*}(x, y)\parallel z  $. The other case  in the above is obvious by the commutativity  of $U^{*}$.  Obviously, $U^{*}$ is a $t$-conorm on $[0,a]$. Assume that  for $z\in I_{a} $, there exists $(x,y)\in[0,a]\times( (0,a] \cap\{l\in L\ |\ l \parallel z\})$  such that   $U^{*}(x, y)\nparallel z  $, that is, $U^{*}(x, y)<  z$. For $(x,y)\in[0,a]\times( (0,a] \cap\{l\in L\ |\ l \parallel z\})$,
 we can obtain that $U^{*}(x, y)\in [x\vee y,a]$.
If $U^{*}(x, y)< z$, then $ y\leq  x\vee y\leq U^{*}(x, y)< z$.  This contradicts with the fact  $y\in \{l\in L\ |\ l \parallel z\}$.
Hence,  $U^{*}(x, y)\parallel z$.

\end{proof}

 In Theorem \ref{th31},  if taking $e=0$, then  $U^{*}$ is a $t$-conorm on $[0,a]$ and $I_{e,a}\cup I_{e}^{a}=\emptyset$. Thus, $I_{e}^{a}\cup(e,a]=(e,a]=(0,a]$  and  the  conditions $(1)$, $(2)$ and  $(4)$  hold  by Lemma \ref{Le32}.  In this case, we can obtain the following proposition.
\begin{proposition}\label{co31}
Let $S$ be a $t$-conorm on $[0,a]$ for $a\in L\setminus\{0,1\}$. Then the function $S_{1}(x,y):L^{2}\rightarrow L$ defined by

$S_{1}(x,y)=\begin{cases}
S(x, y) &\mbox{if } (x,y)\in [0,a]^{2},\\
x\vee y &\mbox{}otherwise,\\
\end{cases}$\\
is a $t$-conorm  if and only if
for   $z\in I_{a}$, if  $(x,y)\in((0,a]\cap\{l\in L\ |\ l \nparallel z\})^{2}$, then $S(x, y)\nparallel z  $.
\end{proposition}

\begin{remark}\label{}
In Proposition \ref{co31}, we  give a sufficient and necessary constraint condition under which $S_{1}$ is a $t$-conorm. Obviously, this condition differs from that  in Theorem \ref{th250}.
 More precisely, our condition is based on  the viewpoint of $t$-conorms; the condition in Theorem \ref{th250} is based on the viewpoint of $L$.


\end{remark}

\begin{example}\label{ex31}
Given a bounded lattice $L_{1}$ drawn in Fig.1. and a uninorm $U^{*}$ on $[0,a]$ shown in Table \ref{Tab:056}. It is clear that $U^{*}$ satisfies the conditions in Theorem \ref{th31} on $L_{1}$. Based on Theorem \ref{th31},  a uninorm $U_{1}$ on $L_{1}$,   shown in Table \ref{Tab:057}, can be obtained.
\end{example}

\begin{minipage}{11pc}
\setlength{\unitlength}{0.75pt}\begin{picture}(600,175)
\put(270,20){\circle{2}}\put(267,12){\makebox(0,0)[l]{\footnotesize$0$}}
\put(270,40){\circle{2}}\put(275,40){\makebox(0,0)[l]{\footnotesize$b$}}
\put(270,60){\circle{2}}\put(274,63){\makebox(0,0)[l]{\footnotesize$e$}}
\put(270,80){\circle{2}}\put(260,80){\makebox(0,0)[l]{\footnotesize$c$}}
\put(270,100){\circle{2}}\put(260,103){\makebox(0,0)[l]{\footnotesize$d$}}
\put(270,120){\circle{2}}\put(275,120){\makebox(0,0)[l]{\footnotesize$a$}}
\put(270,140){\circle{2}}\put(275,140){\makebox(0,0)[l]{\footnotesize$f$}}
\put(270,160){\circle{2}}\put(268,170){\makebox(0,0)[l]{\footnotesize$1$}}
\put(232,80){\circle{2}}\put(222,80){\makebox(0,0)[l]{\footnotesize$l$}}
\put(250,60){\circle{2}}\put(236,60){\makebox(0,0)[l]{\footnotesize$m$}}
\put(250,80){\circle{2}}\put(236,80){\makebox(0,0)[l]{\footnotesize$n$}}
\put(290,120){\circle{2}}\put(295,120){\makebox(0,0)[l]{\footnotesize$s$}}
\put(290,60){\circle{2}}\put(295,60){\makebox(0,0)[l]{\footnotesize$k$}}

\put(270,21){\line(0,1){18}}
\put(270,41){\line(0,1){18}}
\put(270,41){\line(-1,1){19}}
\put(250,61){\line(0,1){18}}
\put(270,61){\line(0,1){18}}
\put(270,81){\line(0,1){18}}
\put(270,81){\line(1,2){19}}
\put(270,101){\line(0,1){18}}
\put(270,121){\line(0,1){18}}
\put(290,121){\line(-1,1){19}}
\put(270,141){\line(0,1){18}}
\put(270,41){\line(1,1){19}}
\put(250,61){\line(-1,1){19}}
\put(250,81){\line(1,1){19}}
\put(290,61){\line(-1,2){19}}
\put(231,81){\line(2,3){39}}

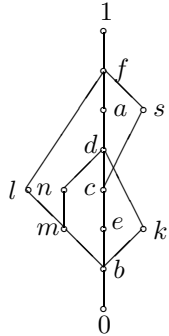
\put(200,-10){\emph{Fig.1. The lattice $L_{1}$}}
\end{picture}
\end{minipage}

\begin{table}[htbp]
\centering
\caption{$U^{*}$ on $[0,a]$.}
\label{Tab:056}\

\begin{tabular}{c c c c c c c c c c}
\hline
  $U^{*}$ & $0$ & $b$ & $e$ & $c$ & $m$ & $n$ & $k$ & $d$ & $a$ \\
\hline
  $0$ & $0$ & $0$ & $0$ & $c$ & $m$ & $n$ & $k$ & $d$ & $a$ \\

  $b$ & $0$ & $b$ & $b$ & $c$ & $m$ & $n$ & $k$ & $d$ & $a$ \\

  $e$ & $0$ & $b$ & $e$ & $c$ & $m$ & $n$ & $k$ & $d$ & $a$ \\

  $c$ & $c$ & $c$ & $c$ & $c$ & $d$ & $d$ & $d$ & $d$ & $a$ \\

  $m$ & $m$ & $m$ & $m$ & $d$ & $m$ & $n$ & $d$ & $d$ & $a$ \\

  $n$ & $n$ & $n$ & $n$ & $d$ & $n$ & $n$ & $d$ & $d$ & $a$ \\

  $k$ & $k$ & $k$ & $k$ & $d$ & $d$ & $d$ & $k$ & $d$ & $a$ \\

  $d$ & $d$ & $d$ & $d$ & $d$ & $d$ & $d$ & $d$ & $d$ & $a$ \\

  $a$ & $a$ & $a$ & $a$ & $a$ & $a$ & $a$ & $a$ & $a$ & $a$ \\
\hline
\end{tabular}
\end{table}

\begin{table}[htbp]
\centering
\caption{$U_{1}$ on $L_{1}$.}
\label{Tab:057}\

\begin{tabular}{c c c c c c c c c c c c c c}
\hline
  $U_{1}$ & $0$ & $b$ & $e$ & $c$ & $m$ & $n$ & $k$ & $d$ & $a$ & $l$ & $s$ & $f$ & $1$ \\
\hline
  $0$ & $0$ & $0$ & $0$ & $c$ & $m$ & $n$ & $k$ & $d$ & $a$ & $l$ & $s$ & $f$ & $1$ \\

  $b$ & $0$ & $b$ & $b$ & $c$ & $m$ & $n$ & $k$ & $d$ & $a$ & $l$ & $s$ & $f$ & $1$ \\

  $e$ & $0$ & $b$ & $e$ & $c$ & $m$ & $n$ & $k$ & $d$ & $a$ & $l$ & $s$ & $f$ & $1$ \\

  $c$ & $c$ & $c$ & $c$ & $c$ & $d$ & $d$ & $d$ & $d$ & $a$ & $f$ & $s$ & $f$ & $1$ \\

  $m$ & $m$ & $m$ & $m$ & $d$ & $m$ & $n$ & $d$ & $d$ & $a$ & $l$ & $f$ & $f$ & $1$ \\

  $n$ & $n$ & $n$ & $n$ & $d$ & $n$ & $n$ & $d$ & $d$ & $a$ & $f$ & $f$ & $f$ & $1$ \\

  $k$ & $k$ & $k$ & $k$ & $d$ & $d$ & $d$ & $k$ & $d$ & $a$ & $f$ & $f$ & $f$ & $1$ \\

  $d$ & $d$ & $d$ & $d$ & $d$ & $d$ & $d$ & $d$ & $d$ & $a$ & $f$ & $f$ & $f$ & $1$ \\

  $a$ & $a$ & $a$ & $a$ & $a$ & $a$ & $a$ & $a$ & $a$ & $a$ & $f$ & $f$ & $f$ & $1$ \\

  $l$ & $l$ & $l$ & $l$ & $f$ & $l$ & $f$ & $f$ & $f$ & $f$ & $l$ & $f$ & $f$ & $1$ \\

  $s$ & $s$ & $s$ & $s$ & $s$ & $f$ & $f$ & $f$ & $f$ & $f$ & $f$ & $s$ & $f$ & $1$ \\

  $f$ & $f$ & $f$ & $f$ & $f$ & $f$ & $f$ & $f$ & $f$ & $f$ & $f$ & $f$ & $f$ & $1$ \\

  $1$ & $1$ & $1$ & $1$ & $1$ & $1$ & $1$ & $1$ & $1$ & $1$ & $1$ & $1$ & $1$ & $1$ \\
\hline
\end{tabular}
\end{table}

\begin{remark}\label{re34KK}
Let $U_{1}$ be a uninorm in Theorem \ref{th31}.\\
$(1)$  $U_{1}$ is disjunctive, i.e., $U_{1}(0,1)=1$.\\
$(2)$ If $a=1$, then $U_{1}=U^{*}$.\\
$(3)$ $U_{1}$ is idempotent if and only if $U^{*}$ is idempotent. This shows that the property of $U_{1}$ are closely related to that of $U^{*}$.\\
$(4)$ $U_{1}\in \mathcal{U}_{max}^{*}$ if and only if $U^{*}\in \mathcal{U}_{max}^{*}$.
\end{remark}

\begin{remark}

Remark \ref{re34KK}(4)  shows that we can easily construct the uninorms, which need  not  belong to the class of $\mathcal{U}_{max}^{*}$. In Theorem \ref{th31},  if $U^{*}\notin \mathcal{U}_{max}^{*}$,  then the uninorm $U_{1}$  does  not belong to $\mathcal{U}_{max}^{*}$. In fact, we can give a example for the uninorm $U^{*}$ on $ [0,a]$ of $L_{2}$,  shown in Table \ref{table1},   satisfying $U^{*}\notin \mathcal{U}_{max}^{*}$ and the conditions in Theorem \ref{th31}. Furthermore,  we can construct  a  uninorm $U_{1}$  such that  $U_{1}\notin \mathcal{U}_{max}^{*}$ by $U^{*}$.
\end{remark}

\begin{minipage}{11pc}
\setlength{\unitlength}{0.75pt}\begin{picture}(600,150)
\put(270,20){\circle{2}}\put(267,12){\makebox(0,0)[l]{\footnotesize$0$}}
\put(270,40){\circle{2}}\put(275,40){\makebox(0,0)[l]{\footnotesize$b$}}
\put(270,60){\circle{2}}\put(274,63){\makebox(0,0)[l]{\footnotesize$e$}}
\put(270,80){\circle{2}}\put(260,82){\makebox(0,0)[l]{\footnotesize$c$}}
\put(270,100){\circle{2}}\put(260,103){\makebox(0,0)[l]{\footnotesize$d$}}
\put(270,120){\circle{2}}\put(275,120){\makebox(0,0)[l]{\footnotesize$a$}}
\put(270,140){\circle{2}}\put(268,150){\makebox(0,0)[l]{\footnotesize$1$}}
\put(232,80){\circle{2}}\put(222,80){\makebox(0,0)[l]{\footnotesize$l$}}
\put(250,60){\circle{2}}\put(236,60){\makebox(0,0)[l]{\footnotesize$m$}}

\put(270,21){\line(0,1){18}}
\put(270,41){\line(0,1){18}}
\put(270,41){\line(-1,1){19}}
\put(250,61){\line(1,1){18}}
\put(270,61){\line(0,1){18}}
\put(270,81){\line(0,1){18}}
\put(270,101){\line(0,1){18}}
\put(270,121){\line(0,1){18}}
\put(250,61){\line(-1,1){19}}
\put(231,81){\line(2,3){39}}

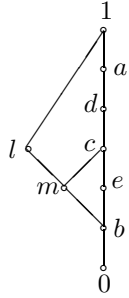
\put(200,-10){\emph{Fig.2. The lattice $L_{2}$}}
\end{picture}
\end{minipage}

\begin{table}[htbp]
\centering
\caption{$U^{*}$ on $[0,a]$.}
\label{table1}\

\begin{tabular}{c c c c c c c c c c}
\hline
  $U^{*}$ & $0$ & $b$ & $e$ & $m$ & $c$ & $d$ & $a$ \\
\hline
  $0$ & $0$ & $0$ & $0$ & $m$ & $c$ & $c$ & $a$ \\

  $b$ & $0$ & $b$ & $b$ & $m$ & $c$ & $c$ & $a$ \\

  $e$ & $0$ & $b$ & $e$ & $m$ & $c$ & $d$ & $a$ \\

  $m$ & $m$ & $m$ & $m$ & $m$ & $a$ & $a$ & $a$ \\

  $c$ & $c$ & $c$ & $c$ & $a$ & $a$ & $a$ & $a$ \\

  $d$ & $c$ & $c$ & $d$ & $a$ & $a$ & $a$ & $a$ \\

  $a$ & $a$ & $a$ & $a$ & $a$ & $a$ & $a$ & $a$ \\
\hline
\end{tabular}
\end{table}


\begin{theorem}\label{th32}
Let $U^{*}$ be a uninorm on $[b,1]$ with a neutral element $e$ for $b\in L\setminus\{0,1\}$. Then the function $U_{2}(x,y):L^{2}\rightarrow L$ defined by

$U_{2}(x,y)=\begin{cases}
U^{*}(x, y) &\mbox{if } (x,y)\in [b,1]^{2},\\
x &\mbox{if } (x,y)\in I_{e,b}\times [e,1],\\
y &\mbox{if } (x,y)\in [e,1]\times I_{e,b},\\
x\wedge y &\mbox{}otherwise,\\
\end{cases}$\\
is a uninorm on $L$ with the neutral element $e\in L $ if and only if $U^{*}$ satisfies  the following conditions:\\
$(1)$ for   $z\in I_{e,b}$, if  $(x,y)\in((I_{e}^{b}\cup [b,e))\cap\{l\in L\ |\ l \nparallel z\})^{2} $, then $U^{*}(x, y)\nparallel z$;\\
$(2)$ for   $z\in I_{e,b}$, if $(x,y)\in[b,1]\times((I_{e}^{b}\cup [b,e))\cap\{l\in L\ |\ l \parallel z\})\cup((I_{e}^{b}\cup [b,e))\cap\{l\in L\ |\ l \parallel z\})\times[b,1]$, then $U^{*}(x, y)\parallel z$;\\
$(3)$ for  $z\in I_{b}^{e}$, if   $(x,y)\in((I_{e}^{b}\cup [b,e))\cap\{l\in L\ |\ l \nparallel z\})^{2} $, then $U^{*}(x, y)\nparallel z  $;\\
$(4)$ for  $z\in I_{b}^{e}$,  if  $(x,y)\in[b,1]\times((I_{e}^{b}\cup [b,e))\cap\{l\in L\ |\ l \parallel z\})\cup((I_{e}^{b}\cup [b,e))\cap\{l\in L\ |\ l \parallel z\})\times[b,1]$, then $U^{*}(x, y)\parallel z$.
\end{theorem}

\begin{proof}
It can be proved immediately by a proof similar to Theorem \ref{th31}.
\end{proof}

\begin{remark}\label{}
If taking $e=b$ in Theorem \ref{th32}, then  $[b,1]=[e,1]$, $I_{e,b}=I_{e}$, $I_{e}^{b}\cup I_{b}^{e}\cup [b,e)=\emptyset$ and $U^{*}$ is a $t$-conorm on $[b,1]$. Moreover, in this case, the conditions $(1)$, $(2)$, $(3)$ and $(4)$ in Theorem \ref{th32} naturally hold.

By the above fact, if taking $e=b$ in Theorem \ref{th31}, then we retrieve the uninorm $U_{s_{1}}: L^{2}\rightarrow L$ constructed by \c{C}ayl{\i}, Kara\c{c}al, and Mesiar (\cite{GD16}, Theorem 1) as follow.

$U_{s_{1}}(x,y)=\begin{cases}
S_{e}(x, y) &\mbox{if } (x,y)\in [e,1]^{2},\\
x &\mbox{if } (x,y)\in I_{e}\times[e,1],\\
y &\mbox{if } (x,y)\in [e,1]\times I_{e},\\
x\wedge y &\mbox{}otherwise.\\
\end{cases}$
\end{remark}


\begin{remark}\label{re36K}
In Theorem \ref{th32}, if taking $e=1$, then $U^{*}$ is a $t$-norm on $[b,1]$ and $I_{e,b}\cup I_{e}^{b}=\emptyset$. Thus,  the conditions $(1)$ and $(2)$ hold and $I_{e}^{b}\cup [b,e)=[b,e)=[b,1)$. In this case,  the condition $(4)$ in Theorem \ref{th32} holds.
\end{remark}

By  Remark \ref{re36K}, If $e=1$ in Theorem \ref{th32}, then   the following proposition holds.

\begin{proposition}\label{co444}
Let $T$ be a $t$-norm on $[b,1]$ for  $b\in L\setminus\{0,1\}$. Then the function $T_{1}(x,y):L^{2}\rightarrow L$ defined by

$T_{1}(x,y)=\begin{cases}
T(x, y) &\mbox{if } (x,y)\in [b,1]^{2},\\
x\wedge y &\mbox{}otherwise,\\
\end{cases}$\\
is a $t$-norm on $L$ if and only if
for   $z\in I_{b}$, if   $(x,y)\in ( [b,1)\cap\{l\in L |\ l \nparallel z\})^{2} $, then $T(x, y)\nparallel z  $.
\end{proposition}

\begin{remark}\label{re}

In Proposition \ref{co444}, we  give a sufficient and necessary condition under which $T_{1}$ is a $t$-norm. Obviously,  this condition  differs from that in Theorem \ref{th250}.
 More precisely, our constraint condition is based on  the viewpoint of $t$-norms; the constraint condition in Theorem \ref{th250} is based on the viewpoint of  $L$.


\end{remark}

\begin{remark}\label{re38K}
Let $U_{2}$ be a uninorm in Theorem \ref{th32}.\\
$(1)$  $U_{2}$ is conjunctive, i.e., $U_{2}(0,1)=0$.\\
$(2)$ If $b=0$, then $U_{2}=U^{*}$.\\
$(3)$ $U_{2}$ is idempotent if and only if $U^{*}$ is idempotent. This shows that the property of $U_{2}$ are closely related to that of $U^{*}$.\\
$(4)$ $U_{2}\in \mathcal{U}_{min}^{*}$ if and only if $U^{*}\in \mathcal{U}_{min}^{*}$.
\end{remark}


\section{Constructing uninorms via given uninorms based on closure and interior operators}

In this section, we mainly construct new uninorms by extending given uninorms based on interior operators and closure operators.

For convenience,
$\mathcal{U}_{\bot}^{*}$ denotes the class of all
uninorms $U$ on $L$ with neutral element $e$ satisfying $U(x, y)\in[0,e]  $ iff $(x,y)\in [0,e]^{2}$.  Similarly, $\mathcal{U}_{\top}^{*}$ denotes  the class of all uninorms $U$ on $L$ with neutral element $e$ satisfying  $U(x, y)\in[e,1]  $ iff $(x,y)\in [e,1]^{2}$.

\begin{theorem}\label{th33}
Let  $U^{*}$ be a uninorm on $[0,a]$ with a neutral element $e$ for $a\in L\setminus\{0,1\}$ and $cl$ be a closure operator. Let $U_{3}(x,y):L^{2}\rightarrow L$ be a  function defined as follow:

$U_{3}(x,y)=\begin{cases}
U^{*}(x, y) &\mbox{if } (x,y)\in [0,a]^{2},\\
x &\mbox{if } (x,y)\in (L\setminus[0,a])\times [0,e],\\
y &\mbox{if } (x,y)\in [0,e]\times (L\setminus[0,a]),\\
cl(x)\vee cl(y) &\mbox{if } (x,y)\in I_{e,a}\times I_{e,a},\\
1 &\mbox{}otherwise.\\
\end{cases}$

$(1)$ Suppose that $cl(x)\vee cl(y)\in I_{e,a}$ for all $x,y\in I_{e,a} $.

 $\mathrm{(i)}$    If  $U^{*}\in\mathcal{U}_{\bot}^{*}$, then $U_{3}$ is a uninorm  with the neutral element $e\in L $  if and only if  $x\parallel y$ for all $x\in I_{e,a}$ and $y\in I_{e}^{a}$.

 $\mathrm{(ii)}$ If  $I_{e,a}\cup I_{a}^{e}\cup (a,1)\neq\emptyset$, then   $U_{3}$ is a uninorm   with the neutral element $e\in L $ if and only if  $U^{*}\in\mathcal{U}_{\bot}^{*}$ and $x\parallel y$ for all $x\in I_{e,a}$ and $y\in I_{e}^{a}$.

$(2)$  Suppose that $cl(x)\vee cl(y)\in (a,1]$ for all $x,y\in I_{e,a} $.

$\mathrm{(i)}$ If $x\parallel y$ for all $x\in I_{e,a}, y\in I_{e}^{a} $ and $U^{*}\in\mathcal{U}_{\bot}^{*}$, then $U_{3}$ is a uninorm   with the neutral element $e\in L $.

$\mathrm{(ii)}$  If $cl(x)\vee cl(y) < 1$ for all $x,y\in I_{e,a}$ and $I_{e,a}\cup I_{a}^{e}\cup (a,1)\neq\emptyset$, then $U_{3}$ is a uninorm  with the neutral element $e\in L $  if and only if  $x\parallel y$ for all $x\in I_{e,a}, y\in I_{e}^{a} $ and $U^{*}\in\mathcal{U}_{\bot}^{*}$.
\end{theorem}

\begin{proof}
$(1)\mathrm{(i)}$ Necessity. Let $U_{3}(x,y)$ be a uninorm with a neutral element $e$ and $cl(x)\vee cl(y)\in I_{e,a}$ for all $x,y\in I_{e,a}$. We prove that $x\parallel y$ for all $x\in I_{e,a}$ and $y\in I_{e}^{a} $.

Assume that there exist $x\in I_{e,a}$ and $y\in I_{e}^{a} $ such that $x\nparallel y$, i.e., $y<x$. Then $U_{3}(x,y)=1$ and $U_{3}(x,x)=cl(x)\vee cl(x)=cl(x) $. Since $cl(x) <1$, the  increasingness property of $U_{3}$ is violated. Thus $x\parallel y$   for all $x\in I_{e,a}$ and $y\in I_{e}^{a} $.

Sufficiency. By the definition of $U_{3}$, $U_{3}$ is commutative  and  $e$  is the neutral element of $U_{3}$. Thus, we only need to prove the increasingness and the associativity of $U_{3}$.

I. Increasingness: We prove that if $x\leq y$, then $U_{3}(x,z)\leq U_{3}(y,z)$ for all $z\in L$. It is easy to verify  that $U_{3}(x,z)\leq U_{3}(y,z)$ if both $x$ and $y$ belong to one of the  intervals $ [0,e], I_{e}^{a}, (e,a], I_{a}^{e}, I_{e,a}$ or $(a,1]$ for all $z\in L$. The proof is split into all possible cases.

1. $x\in [0,e]$

\ \ \ 1.1. $y\in I_{e}^{a}\cup (e,a]$

\ \ \ \ \ \ 1.1.1. $z\in [0,e]\cup I_{e}^{a}\cup (e,a]$

\ \ \ \ \ \ \ \ \ \ \ \ $U_{3}(x,z)=U^{*}(x,z)\leq U^{*}(y,z)=U_{3}(y,z)$

\ \ \ \ \ \ 1.1.2. $z\in I_{a}^{e}\cup I_{e,a}\cup(a,1]$

\ \ \ \ \ \ \ \ \ \ \ \ $U_{3}(x,z)=z\leq 1=U_{3}(y,z)$

\ \ \ 1.2. $y\in I_{a}^{e}\cup (a,1]$

\ \ \ \ \ \ 1.2.1. $z\in [0,e]$

\ \ \ \ \ \ \ \ \ \ \ \ $U_{3}(x,z)=U^{*}(x,z)\leq x<y =U_{3}(y,z)$

\ \ \ \ \ \ 1.2.2. $z\in I_{e}^{a}\cup (e,a]$

\ \ \ \ \ \ \ \ \ \ \ \ $U_{3}(x,z)=U^{*}(x,z)\leq a<1=U_{3}(y,z)$

\ \ \ \ \ \ 1.2.3. $z\in I_{a}^{e}\cup I_{e,a}\cup(a,1]$

\ \ \ \ \ \ \ \ \ \ \ \ $U_{3}(x,z)=z\leq 1=U_{3}(y,z)$

\ \ \ 1.3. $y\in I_{e,a}$

\ \ \ \ \ \ 1.3.1. $z\in [0,e]$

\ \ \ \ \ \ \ \ \ \ \ \ $U_{3}(x,z)=U^{*}(x,z)\leq x<y =U_{3}(y,z)$

\ \ \ \ \ \ 1.3.2. $z\in I_{e}^{a}\cup (e,a]$

\ \ \ \ \ \ \ \ \ \ \ \ $U_{3}(x,z)=U^{*}(x,z)\leq a<1=U_{3}(y,z)$

\ \ \ \ \ \ 1.3.3. $z\in I_{a}^{e} \cup(a,1]$

\ \ \ \ \ \ \ \ \ \ \ \ $U_{3}(x,z)=z\leq 1=U_{3}(y,z)$

\ \ \ \ \ \ 1.3.4. $z\in   I_{e,a} $

\ \ \ \ \ \ \ \ \ \ \ \ $U_{3}(x,z)=z\leq cl(y)\vee cl(z)=U_{3}(y,z)$

2. $x\in I_{e}^{a}$

\ \ \ 2.1. $y\in (e,a]$

\ \ \ \ \ \ 2.1.1. $z\in [0,e]\cup I_{e}^{a}\cup (e,a]$

\ \ \ \ \ \ \ \ \ \ \ \ $U_{3}(x,z)=U^{*}(x,z)\leq U^{*}(y,z)=U_{3}(y,z)$

\ \ \ \ \ \ 2.1.2. $z\in I_{a}^{e}\cup I_{e,a}\cup (a,1]$

\ \ \ \ \ \ \ \ \ \ \ \ $U_{3}(x,z)=1=U_{3}(y,z)$

\ \ \ 2.2. $y\in I_{a}^{e} \cup(a,1]$

\ \ \ \ \ \ 2.2.1. $z\in [0,e]$

\ \ \ \ \ \ \ \ \ \ \ \ $U_{3}(x,z)=U^{*}(x,z)\leq x<y =U_{3}(y,z)$

\ \ \ \ \ \ 2.2.2. $z\in I_{e}^{a}\cup (e,a]$

\ \ \ \ \ \ \ \ \ \ \ \ $U_{3}(x,z)=U^{*}(x,z)\leq a<1=U_{3}(y,z)$

\ \ \ \ \ \ 2.2.3. $z\in I_{a}^{e}\cup I_{e,a}\cup (a,1]$

\ \ \ \ \ \ \ \ \ \ \ \ $U_{3}(x,z)=1=U_{3}(y,z)$

3. $x\in (e,a], y\in I_{a}^{e}\cup(a,1]$

\ \ \ 3.1. $z\in [0,e]$

\ \ \ \ \ \ \ \ \ \ \ \ $U_{3}(x,z)=U^{*}(x,z)\leq x<y =U_{3}(y,z)$

\ \ \ 3.2. $z\in I_{e}^{a}\cup (e,a]$

\ \ \ \ \ \ \ \ \ \ \ \ $U_{3}(x,z)=U^{*}(x,z)\leq a<1=U_{3}(y,z)$

\ \ \ 3.3. $z\in I_{a}^{e}\cup I_{e,a}\cup(a,1]$

\ \ \ \ \ \ \ \ \ \ \ \ $U_{3}(x,z)=1=U_{3}(y,z)$

4. $x\in I_{a}^{e}, y\in (a,1]$

\ \ \ 4.1. $z\in [0,e]$

\ \ \ \ \ \ \ \ \ \ \ \ $U_{3}(x,z)=x\leq y =U_{3}(y,z)$

\ \ \ 4.2. $z\in I_{e}^{a}\cup(e,a]\cup I_{a}^{e}\cup I_{e,a}\cup(a,1]$

\ \ \ \ \ \ \ \ \ \ \ \ $U_{3}(x,z)=1=U_{3}(y,z)$

5. $x\in I_{e,a}, y\in I_{a}^{e}\cup(a,1]$

\ \ \ 5.1. $z\in [0,e]$

\ \ \ \ \ \ \ \ \ \ \ \ $U_{3}(x,z)=x< y =U_{3}(y,z)$

\ \ \ 5.2. $z\in I_{e}^{a}\cup(e,a]\cup I_{a}^{e}\cup (a,1]$

\ \ \ \ \ \ \ \ \ \ \ \ $U_{3}(x,z)=1=U_{3}(y,z)$

\ \ \ 5.3. $z\in I_{e,a} $

\ \ \ \ \ \ \ \ \ \ \ \ $U_{3}(x,z)=cl(x)\vee cl(z)<1=U_{3}(y,z)$

II. Associativity: We demonstrate that $U_{3}(x,U_{3}(y,z))=U_{3}(U_{3}(x,y),z)$ for all $x,y,z\in L$. By Proposition  \ref{pro2.1},  we just consider the following cases.

1. If $x,y,z\in [0,e]\cup I_{e}^{a}\cup (e,a]$, then $U_{3}(x,U_{3}(y,z))
=U_{3}(U_{3}(x,y),z)
=U_{3}(y,U_{3}(x,z))$ for $U^{*}$ is associative.

2. If $x,y,z\in I_{a}^{e}\cup (a,1]$, then $U_{3}(x,U_{3}(y,z))=U_{3}(x,1)=1=U_{3}(1,z)=U_{3}(U_{3}(x,y),z)$.

3. If $x,y,z\in  I_{e,a} $, then $U_{3}(x,U_{3}(y,z))=U_{3}(x,cl(y)\vee cl(z))=cl(x)\vee cl(y)\vee cl(z)=U_{3}(cl(x)\vee cl(y),z)= U_{3}(U_{3}(x,y),z)$.

4. If $x,y\in [0,e]$ and $z\in I_{a}^{e}\cup I_{e,a}\cup  (a,1]$, then $U_{3}(x,U_{3}(y,z))=U_{3}(x,z)=z =U_{3}(U^{*}(x,y),z)=U_{3}(U_{3}(x,y),z) $.

5. If $x,y\in I_{e}^{a}\cup (e,a]$ and $z\in I_{a}^{e}\cup I_{e,a}\cup  (a,1]$, then $U_{3}(x,U_{3}(y,z))=U_{3}(x,1)=1=U_{3}(U^{*}(x,y),z)$ $=U_{3}(U_{3}(x,y),z) $ and $U_{3}(y,U_{3}(x,z))=U_{3}(y,1)=1$. Thus $U_{3}(x,U_{3}(y,z))$ $=U_{3}(U_{3}(x,y),z)=U_{3}(y,U_{3}(x,z))$.

6. If $x,y\in I_{a}^{e}$ and $z\in  I_{e,a}$, then $U_{3}(x,U_{3}(y,z))=U_{3}(x,1)=1=U_{3}(1,z)=U_{3}(U_{3}(x,y),z)$.

7. If $x,y\in I_{e,a}$ and $z\in (a,1]$, then $U_{3}(x,U_{3}(y,z))=U_{3}(x,1)=1=U_{3}(cl(x)\vee cl(y),z)=U_{3}(U_{3}(x,y),z)$.

8. If $x\in [0,e]$ and $y,z\in I_{a}^{e}\cup (a,1] $, then $U_{3}(x,U_{3}(y,z))=U_{3}(x,1)=1=U_{3}(y,z)=U_{3}(U_{3}(x,y),z)$ and $U_{3}(y,U_{3}(x,z))=U_{3}(y,z)=1$. Thus $U_{3}(x,U_{3}(y,z))$ $=U_{3}(U_{3}(x,y),z)=U_{3}(y,U_{3}(x,z))$.

9. If $x\in [0,e]$ and $y,z\in   I_{e,a}$, then $U_{3}(x,U_{3}(y,z))=U_{3}(x,cl(y)\vee cl(z))=cl(y)\vee cl(z)=U_{3}(y,z)= U_{3}(U_{3}(x,y),z)$.

10. If $x\in I_{e}^{a}\cup (e,a]$ and $y,z\in I_{a}^{e} \cup  (a,1] $, then $U_{3}(x,U_{3}(y,z))=U_{3}(x,1)=1=U_{3}(1,z)=U_{3}(U_{3}(x,y),z)$ and $U_{3}(y,U_{3}(x,z))=U_{3}(y,1)=1$. Thus $U_{3}(x,U_{3}(y,z))=U_{3}(U_{3}(x,y),z)=U_{3}(y,U_{3}(x,z))$.

11. If $x\in I_{e}^{a}\cup (e,a]$ and $y,z\in  I_{e,a}$, then $U_{3}(x,U_{3}(y,z))=U_{3}(x,cl(y)\vee cl(z))=1=U_{3}(1,z)=U_{3}(U_{3}(x,y),z)$.

12. If $x\in I_{a}^{e}$ and $y,z\in I_{e,a}$, then $U_{3}(x,U_{3}(y,z))=U_{3}(x,cl(y)\vee cl(z))=1=U_{3}(1,z)=U_{3}(U_{3}(x,y),z)$.

13. If $x\in I_{e,a}$ and $y,z\in (a,1]$, then $U_{3}(x,U_{3}(y,z))=U_{3}(x,1)=1=U_{3}(1,z)=U_{3}(U_{3}(x,y),z)$. Thus $U_{3}(x,U_{3}(y,z))=U_{3}(U_{3}(x,y),z)$.

14. If $x\in [0,e],y\in I_{e}^{a}\cup (e,a]$ and $z\in I_{a}^{e}\cup I_{e,a}\cup  (a,1] $, then $U_{3}(x,U_{3}(y,z))=U_{3}(x,1)=1=U_{3}(U^{*}(x,y),z)=U_{3}(U_{3}(x,y),z)$ and $U_{3}(y,U_{3}(x,z))=U_{3}(y,z)=1$. Thus $U_{3}(x,U_{3}(y,z))$ $=U_{3}(U_{3}(x,y),z)$ $=U_{3}(y,U_{3}(x,z))$.

15. If $x\in [0,e],y\in I_{a}^{e}$ and $z\in I_{e,a}$, then $U_{3}(x,U_{3}(y,z))=U_{3}(x,1)=1=U_{3}(y,z)=U_{3}(U_{3}(x,y),z)$ and $U_{3}(y,U_{3}(x,z))=U_{3}(y,z)=1$. Thus $U_{3}(x,U_{3}(y,z))=U_{3}(U_{3}(x,y),z)=U_{3}(y,U_{3}(x,z))$.

16. If $x\in [0,e],y\in I_{e,a}$ and $z\in (a,1]$, then $U_{3}(x,U_{3}(y,z))=U_{3}(x,1)=1=U_{3}(y,z)=U_{3}(U_{3}(x,y),z)$ and $U_{3}(y,U_{3}(x,z))=U_{3}(y,z)=1$. Thus $U_{3}(x,U_{3}(y,z))$ $=U_{3}(U_{3}(x,y),z)=U_{3}(y,U_{3}(x,z))$.

17. If $x\in I_{e}^{a}\cup (e,a],y\in I_{a}^{e}$ and $z\in I_{e,a}$, then $U_{3}(x,U_{3}(y,z))=U_{3}(x,1)=1=U_{3}(1,z)=U_{3}(U_{3}(x,y),z)$ and $U_{3}(y,U_{3}(x,z))=U_{3}(y,1)=1$. Thus $U_{3}(x,U_{3}(y,z))=U_{3}(U_{3}(x,y),z)=U_{3}(y,U_{3}(x,z))$.

18. If $x\in I_{e}^{a}\cup (e,a],y\in I_{e,a}$ and $z\in (a,1]$, then $U_{3}(x,U_{3}(y,z))=U_{3}(x,1)=1=U_{3}(1,z)=U_{3}(U_{3}(x,y),z)$ and $U_{3}(y,U_{3}(x,z))=U_{3}(y,1)=1$. Thus $U_{3}(x,U_{3}(y,z))=U_{3}(U_{3}(x,y),z)$ $=U_{3}(y,U_{3}(x,z))$.

19. If $x\in I_{a}^{e},y\in I_{e,a}$ and $z\in (a,1]$, then $U_{3}(x,U_{3}(y,z))=U_{3}(x,1)=1=U_{3}(1,z)=U_{3}(U_{3}(x,y),z)$ and $U_{3}(y,U_{3}(x,z))=U_{3}(y,1)=1$. Thus $U_{3}(x,U_{3}(y,z))=U_{3}(U_{3}(x,y),z)$ $=U_{3}(y,U_{3}(x,z))$.

$(1)\mathrm{(ii)}$  We just prove that If $I_{e,a}\cup I_{a}^{e}\cup (a,1)\neq\emptyset$, then the condition  $U^{*}\in\mathcal{U}_{\bot}^{*}$ is necessary.

It is obvious that if $(x,y)\in [0,e]^{2}$, then $U^{*}(x, y)\in[0,e]$. Thus, we only need to prove that if $U^{*}(x, y)\in[0,e]  $, then $(x,y)\in [0,e]^{2}$.
 The proof can be  split into all possible cases.

$a$. $U^{*}(x, y)\notin[0,e]  $ for all $(x,y)\in [0,e]\times (I_{e}^{a}\cup (e,a])\cup (I_{e}^{a}\cup (e,a])\times [0,e]\cup I_{e}^{a}\times I_{e}^{a}$.

Now we give the proof of $U^{*}(x, y)\notin[0,e]  $ for all $(x,y)\in [0,e]\times (I_{e}^{a}\cup (e,a])\cup I_{e}^{a}\times I_{e}^{a}$, and the other case is obvious by the commutativity  of $U^{*}$. Assume that   there exists $(x,y)\in [0,e]\times (I_{e}^{a}\cup (e,a])\cup I_{e}^{a}\times I_{e}^{a}$ such that  $U^{*}(x, y)\in[0,e]$.  If $z\in I_{e,a}\cup I_{a}^{e}\cup (a,1)$, then $U_{3}(x,U_{3}(y,z))=U_{3}(x,1)=1$ and $U_{3}(U_{3}(x,y),z)=U_{3}(U^{*}(x,y),z)=z$. Since $z<1$, the associativity of $U_{3}$ is violated. Thus $U^{*}(x, y)\notin[0,e]  $ for all $(x,y)\in [0,e)\times (I_{e}^{a}\cup (e,a])\cup (I_{e}^{a}\cup (e,a])\times [0,e)\cup I_{e}^{a}\times I_{e}^{a}$.

b. $U^{*}(x, y)\notin[0,e]  $ for all $(x,y)\in (e,a]^{2}\cup (e,a]\times I_{e}^{a}\cup I_{e}^{a}\times (e,a]$.

Now we just prove that  $U^{*}(x, y)\notin[0,e]  $ for all $(x,y)\in (e,a]^{2}\cup (e,a]\times I_{e}^{a}$, and the other case  is obvious by the commutativity  of $U^{*}$. By the increasingness  of $U^{*}$,  we can obtain that $y =U^{*}(e, y)\leq U^{*}(x, y)$. Since $y\in I_{e}^{a} \cup (e,a]$,  $U^{*}(x, y)\notin [0,e]$. Thus $U^{*}(x, y)\notin[0,e]  $ for all $(x,y)\in (e,a]^{2}\cup (e,a]\times I_{e}^{a}\cup I_{e}^{a}\times (e,a]$.

$(2)\mathrm{(i)}$ By the definition of $U_{3}$,   $U_{3}$ is commutative  and  $e$  is the neutral element of $U_{3}$. Thus, we only need to show the increasingness and the associativity of $U_{3}$.

I. Increasingness: We prove that if $x\leq y$, then $U_{3}(x,z)\leq U_{3}(y,z)$ for all $z\in L$.
Next, it is enough to check  the cases that are different from those in the proof of Theorem \ref{th33}$(1)\mathrm{(i)}$.

1. $x\in [0,e], y\in I_{e,a}$

\ \ \ 1.1. $z\in   I_{e,a} $

\ \ \ \ \ \ \ \ \ \ \ \ $U_{3}(x,z)=z< cl(y)\vee cl(z)=U_{3}(y,z)$

2. $x\in I_{e,a}$

\ \ \ 2.1. $y\in I_{e,a}$

\ \ \ \ \ \ 2.1.1. $z\in I_{e,a} $

\ \ \ \ \ \ \ \ \ \ \ \ $U_{3}(x,z)=cl(x)\vee cl(z) \leq cl(y)\vee cl(z)=U_{3}(y,z)$

\ \ \ 2.2. $y\in I_{a}^{e}\cup(a,1]$

\ \ \ \ \ \ 2.2.1. $z\in I_{e,a} $

\ \ \ \ \ \ \ \ \ \ \ \ $U_{3}(x,z)=cl(x)\vee cl(z)\leq 1=U_{3}(y,z)$

II. Associativity: It can be shown that $U_{3}(x,U_{3}(y,z))=U_{3}(U_{3}(x,y),z)$  for all $x,y,z\in L$. By Proposition  \ref{pro2.1} and the proof of Theorem \ref{th33}$(1)(i)$, we just  check the cases that are different from the cases in the proof of Theorem \ref{th33}$(1)\mathrm{(i)}$.

1. If $x,y,z\in  I_{e,a} $, then $U_{3}(x,U_{3}(y,z))=U_{3}(x,cl(y)\vee cl(z))=1=U_{3}(cl(x)\vee cl(y),z)=U_{3}(U_{3}(x,y),z)$.

2. If $x\in [0,e]$ and $y,z\in   I_{e,a}$, then $U_{3}(x,U_{3}(y,z))=U_{3}(x,cl(y)\vee cl(z))=cl(y)\vee cl(z)=U_{3}(y,z)=U_{3}(U_{3}(x,y),z) $.

3. If $x\in I_{e}^{a}\cup (e,a]\cup I_{a}^{e}$ and $y,z\in  I_{e,a}$, then $U_{3}(x,U_{3}(y,z))=U_{3}(x,cl(y)\vee cl(z))=1=U_{3}(1,z)=U_{3}(U_{3}(x,y),z)$.

4. If $x,y\in I_{e,a}$ and $z\in (a,1]$, then $U_{3}(x,U_{3}(y,z))=U_{3}(x,1)=1=U_{3}(cl(x)\vee cl(y),z)=U_{3}(U_{3}(x,y),z)$.

$(2)\mathrm{(ii)}$ In the following, we only prove that if $cl(x)\vee cl(y)\in(a,1)$ for all $x,y\in I_{e,a} $ and $I_{e,a}\cup I_{a}^{e}\cup (a,1)\neq\emptyset$, then  $x\parallel y$ for all $x\in I_{e,a}, y\in I_{e}^{a} $ and $U^{*}\in\mathcal{U}_{\bot}^{*}$ are necessary.
First,  assume that there exist $x\in I_{e,a}$ and $y\in I_{e}^{a} $ such that $x\nparallel y$, i.e., $y<x$. Then $U_{3}(x,x)=cl(x)\vee cl(x)$ and $U_{3}(x,y)=1$. Since $cl(x)\vee cl(x)<1$, the increasingness of $U_{3}$ is violated. Thus $x\parallel y$ for all $x\in I_{e,a}$ and $y\in I_{e}^{a}$.
Then we can obtain that $U^{*}\in\mathcal{U}_{\bot}^{*}$ is necessary  by the proof similar to Theorem \ref{th33}$(1)\mathrm{(ii)}$.
\end{proof}

\begin{remark}
If taking $e=0$ in Theorem \ref{th33}, then $[e,a]=[0,a]$, $I_{a}^{e}=I_{a}$, $I_{e,a}\cup I_{e}^{a}\cup [0,e)=\emptyset$ and $U^{*}$ is a $t$-conorm on $[0,a]$. Moreover, based on the  above  case,  the following conditions naturally hold.\\
$(1)$ $cl(x)\vee cl(y)\in I_{e,a}$ or $cl(x)\vee cl(y)\in (a,1]$ for all $x,y\in I_{e,a} $. \\
$(2)$ $U^{*}\in\mathcal{U}_{\bot}^{*}$ and $x\parallel y$ for all $x\in I_{e,a}$ and $y\in I_{e}^{a} $.

By the above fact, if taking $e=0$ in Theorem \ref{th33}, then we retrieve the  $t$-conorm $S_{2}^{*}: L^{2}\rightarrow L$ constructed by \c{C}ayl{\i} (\cite{GD018}, Theorem 1) as follow.

$S_{2}^{*}(x,y)=\begin{cases}
S_{e}(x, y) &\mbox{if } (x,y)\in [0,a]^{2},\\
x &\mbox{if } (x,y)\in (I_{a}\cup(a,1])\times \{0\},\\
y &\mbox{if } (x,y)\in \{0\}\times (I_{a}\cup(a,1]),\\
1 &\mbox{}otherwise.\\
\end{cases}$
\end{remark}

\begin{remark}
If taking $e=a$ in Theorem \ref{th33}, then $[0,a]=[0,e]$, $I_{e,a}=I_{e}$, $I_{a}^{e}\cup I_{e}^{a}\cup (e,a]=\emptyset$ and $U^{*}$ is a $t$-norm on $[0,a]$. Moreover, in this case,
$U^{*}\in\mathcal{U}_{\bot}^{*}$ and $x\parallel y$ for all $x\in I_{e,a}$ and $y\in I_{e}^{a} $ naturally hold.

By the above fact, if taking $e=a$ in Theorem \ref{th33}, then we retrieve the   uninorm $U_{cl}: L^{2}\rightarrow L$ constructed by Zhao and Wu (\cite{BZ21}, Proposition 3.1) as follow.

$U_{cl}(x,y)=\begin{cases}
T_{e}(x, y) &\mbox{if } (x,y)\in [0,e]^{2},\\
x &\mbox{if } (x,y)\in I_{e}\times [0,e]\cup(e,1]\times[0,e],\\
y &\mbox{if } (x,y)\in [0,e]\times I_{e}\cup [0,e]\times(e,1],\\
cl(x)\vee cl(y) &\mbox{if } (x,y)\in I_{e}\times I_{e},\\
1 &\mbox{}otherwise.\\
\end{cases}$
\end{remark}

If we put $cl(x)=x$ for $x\in L$ in Theorem \ref{th33}, then we obtain the following corollary.

\begin{corollary}\label{coro331}
Let  $a\in L\setminus\{0,1\}$ and $U^{*}$ be a uninorm on $[0,a]$ with a neutral element $e$. Let $U_{31}(x,y):L^{2}\rightarrow L$ be a  function defined as follow:

$U_{31}(x,y)=\begin{cases}
U^{*}(x, y) &\mbox{if } (x,y)\in [0,a]^{2},\\
x &\mbox{if } (x,y)\in (L\setminus[0,a])\times [0,e],\\
y &\mbox{if } (x,y)\in [0,e]\times (L\setminus[0,a]),\\
x\vee y &\mbox{if } (x,y)\in I_{e,a}\times I_{e,a},\\
1 &\mbox{}otherwise.\\
\end{cases}$

$(1)$ Suppose that $x\vee y\in I_{e,a}$ for all $x,y\in I_{e,a} $.

$\mathrm{(i)}$    If  $U^{*}\in\mathcal{U}_{\bot}^{*}$, then $U_{31}$ is a uninorm   with the neutral element $e\in L $ if and only if $x\parallel y$ for all $x\in I_{e,a}$ and $y\in I_{e}^{a} $.

$\mathrm{(ii)}$ If  $I_{e,a}\cup I_{a}^{e}\cup (a,1)\neq\emptyset$, then $U_{31}$ is a uninorm   with the neutral element $e\in L $ if and only if  $U^{*}\in\mathcal{U}_{\bot}^{*}$ and $x\parallel y$ for all $x\in I_{e,a}$ and $y\in I_{e}^{a} $.

$(2)$  Suppose that $x\vee y\in (a,1]$ for all $x,y\in I_{e,a} $.

$\mathrm{(i)}$ If $x\parallel y$ for all $x\in I_{e,a}, y\in I_{e}^{a} $ and $U^{*}\in\mathcal{U}_{\bot}^{*}$, then $U_{31}$ is a uninorm  with the neutral element $e\in L $.

$\mathrm{(ii)}$  If $x\vee y<1$ for all $x,y\in I_{e,a}$ and $I_{e,a}\cup I_{a}^{e}\cup (a,1)\neq\emptyset$,  then $U_{31}$ is a uninorm   with the neutral element $e\in L $  if and only if $x\parallel y$ for all $x\in I_{e,a}, y\in I_{e}^{a} $ and $U^{*}\in\mathcal{U}_{\bot}^{*}$.
\end{corollary}

\begin{remark}

If taking $e=a$ in Corollary \ref{coro331}, then we retrieve the  uninorm $U_{t_{2}}: L^{2}\rightarrow L$ constructed by \c{C}ayl{\i} and Kara\c{c}al (\cite{GD17}, Theorem 3.1).

$U_{t_{2}}(x,y)=\begin{cases}
T_{e}(x, y) &\mbox{if } (x,y)\in [0,e]^{2},\\
x &\mbox{if } (x,y)\in I_{e}\times [0,e]\cup(e,1]\times[0,e],\\
y &\mbox{if } (x,y)\in [0,e]\times I_{e }\cup [0,e]\times(e,1],\\
x\vee y &\mbox{if } (x,y)\in I_{e}\times I_{e},\\
1 &\mbox{}otherwise.\\
\end{cases}$
\end{remark}

If we put $cl(x)=x\vee a$ for $x\in L$ in Theorem \ref{th33}, then we obtain the following corollary based on Theorem \ref{th33}$(2)$.

\begin{corollary}\label{coro332}
Let $U^{*}$ be a uninorm on $[0,a]$ with a neutral element $e$ for $a\in L\setminus\{0,1\}$.  Let $U_{32}(x,y):L^{2}\rightarrow L$ be a  function defined as follows:

$U_{32}(x,y)=\begin{cases}
U^{*}(x, y) &\mbox{if } (x,y)\in [0,a]^{2},\\
x &\mbox{if } (x,y)\in (L\setminus[0,a])\times [0,e],\\
y &\mbox{if } (x,y)\in [0,e]\times (L\setminus[0,a]),\\
x\vee y\vee a &\mbox{if } (x,y)\in I_{e,a}\times I_{e,a},\\
1 &\mbox{}otherwise.\\
\end{cases}$

$(1)$ If $x\parallel y$ for all $x\in I_{e,a}, y\in I_{e}^{a} $ and $U^{*}\in\mathcal{U}_{\bot}^{*}$, then $U_{32}$ is a uninorm  with the neutral element $e\in L $.

$(2)$  If  $I_{e,a}\cup I_{a}^{e}\cup (a,1)\neq\emptyset$ and $x\vee y\vee a<1$ for all $x,y\in I_{e,a} $, then $U_{32}$ is a uninorm   with the neutral element $e\in L $ if and only if  $x\parallel y$ for all $x\in I_{e,a}, y\in I_{e}^{a} $ and $U^{*}\in\mathcal{U}_{\bot}^{*}$.
\end{corollary}

\begin{remark}
If taking $e=a$ in Corollary \ref{coro332}, then we retrieve the  uninorm $U_{t_{3}}: L^{2}\rightarrow L$ constructed by \c{C}ayl{\i} (\cite{GD018b}, Theorem 2.23).

$U_{t_{3}}(x,y)=\begin{cases}
T_{e}(x, y) &\mbox{if } (x,y)\in [0,e]^{2},\\
x &\mbox{if } (x,y)\in (L\setminus [0,e])\times [0,e],\\
y &\mbox{if } (x,y)\in [0,e]\times (L\setminus [0,e]),\\
x\vee y\vee e &\mbox{if } (x,y)\in I_{e}\times I_{e},\\
1 &\mbox{}otherwise.\\
\end{cases}$
\end{remark}

\begin{example}
Given a bounded lattice $L_{3}$ depicted in Fig.3., a uninorm $U^{*}$ on $[0,a]$ shown in Table \ref{Tab:01} and a closure operator $cl$ on $L_{3}$ defined by $cl(x)=x$ for all $x\in L_{3}$.  It is easy to see that $L_{3}$, $U^{*}$ and $cl$ satisfy the conditions in Theorem \ref{th33}$(1)\mathrm{(i)}$ on $L_{3}$. By  Theorem \ref{th33}, we can obtain a uninorm $U_{3}$ on $L_{4}$ with the neutral element $e$, shown in Table \ref{Tab:02}.
\end{example}

\begin{minipage}{11pc}
\setlength{\unitlength}{0.75pt}\begin{picture}(600,160)
\put(270,20){\circle{2}}\put(267,12){\makebox(0,0)[l]{\footnotesize$0$}}
\put(270,40){\circle{2}}\put(260,40){\makebox(0,0)[l]{\footnotesize$b$}}
\put(270,60){\circle{2}}\put(255,60){\makebox(0,0)[l]{\footnotesize$e$}}
\put(270,80){\circle{2}}\put(260,81){\makebox(0,0)[l]{\footnotesize$c$}}
\put(270,100){\circle{2}}\put(260,101){\makebox(0,0)[l]{\footnotesize$a$}}
\put(270,120){\circle{2}}\put(275,123){\makebox(0,0)[l]{\footnotesize$d$}}
\put(270,140){\circle{2}}\put(267,148){\makebox(0,0)[l]{\footnotesize$1$}}
\put(231,101){\circle{2}}\put(222,103){\makebox(0,0)[l]{\footnotesize$l$}}
\put(250,81){\circle{2}}\put(237,83){\makebox(0,0)[l]{\footnotesize$n$}}
\put(211,81){\circle{2}}\put(200,81){\makebox(0,0)[l]{\footnotesize$t$}}
\put(231,61){\circle{2}}\put(216,59){\makebox(0,0)[l]{\footnotesize$m$}}
\put(290,100){\circle{2}}\put(295,100){\makebox(0,0)[l]{\footnotesize$s$}}
\put(290,60){\circle{2}}\put(295,60){\makebox(0,0)[l]{\footnotesize$k$}}

\put(270,21){\line(0,1){18}}
\put(270,21){\line(-1,1){39}}
\put(270,41){\line(0,1){18}}
\put(230,61){\line(-1,1){19}}
\put(270,61){\line(0,1){18}}
\put(250,82){\line(-1,1){18}}
\put(270,81){\line(0,1){18}}
\put(270,81){\line(1,1){19}}
\put(211,82){\line(1,1){19}}
\put(230,62){\line(1,1){19}}
\put(270,101){\line(0,1){18}}
\put(290,101){\line(-1,1){19}}
\put(270,121){\line(0,1){18}}
\put(270,41){\line(1,1){19}}
\put(290,61){\line(-1,1){19}}
\put(231,102){\line(2,1){38}}

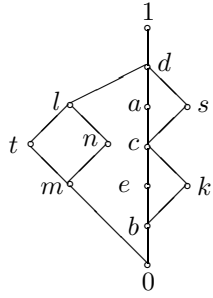
\put(200,-10){\emph{Fig.3. The lattice $L_{3}$.}}
\end{picture}
\end{minipage}

\begin{table}[htbp]
\centering
\caption{$U^{*}$ on $[0,a]$.}
\label{Tab:01}\

\begin{tabular}{c c c c c c c c c}
\hline
  $U^{*}$ & $0$ & b & $e$ & $k$ & $c$ & $a$ \\
\hline
  $0$ & $0$ & $0$ & $0$ & $k$ & $c$ & $a$ \\

  b & $0$ & b & b & $k$ & $c$ & $a$ \\

  $e$ & $0$ & b & $e$ & $k$ & $c$ & $a$ \\

  $k$ & $k$ & $k$ & $k$ & $k$ & $c$ & $a$ \\

  $c$ & $c$ & $c$ & $c$ & $c$ & $c$ & $a$ \\

  $a$ & $a$ & $a$ & $a$ & $a$ & $a$ & $a$ \\
\hline
\end{tabular}
\end{table}

\begin{table}[htbp]
\centering
\caption{$U_{3} $ on $L_{3}$.}
\label{Tab:02}\

\begin{tabular}{c c c c c c c c c c c c c c}
\hline
  $U_{3}$ & $0$ & $b$ & $e$ & $k$ & $c$ & $a$ & $m$ & $t$ & $n$ & $l$ & $s$ & $d$ & $1$ \\
\hline
  $0$ & $0$ & $0$ & $0$ & $k$ & $c$ & $a$ & $m$ & $t$ & $n$ & $l$ & $s$ & $d$ & $1$ \\

  $b$ & $0$ & $b$ & $b$ & $k$ & $c$ & $a$ & $m$ & $t$ & $n$ & $l$ & $s$ & $d$ & $1$ \\

  $e$ & $0$ & $b$ & $e$ & $k$ & $c$ & $a$ & $m$ & $t$ & $n$ & $l$ & $s$ & $d$ & $1$ \\

  $k$ & $k$ & $k$ & $k$ & $k$ & $c$ & $a$ & $1$ & $1$ & $1$ & $1$ & $1$ & $1$ & $1$ \\

  $c$ & $c$ & $c$ & $c$ & $c$ & $c$ & $a$ & $1$ & $1$ & $1$ & $1$ & $1$ & $1$ & $1$ \\

  $a$ & $a$ & $a$ & $a$ & $a$ & $a$ & $a$ & $1$ & $1$ & $1$ & $1$ & $1$ & $1$ & $1$ \\

  $m$ & $m$ & $m$ & $m$ & $1$ & $1$ & $1$ & $m$ & $t$ & $n$ & $l$ & $1$ & $1$ & $1$ \\

  $t$ & $t$ & $t$ & $t$ & $1$ & $1$ & $1$ & $t$ & $t$ & $l$ & $l$ & $1$ & $1$ & $1$ \\

  $n$ & $n$ & $n$ & $n$ & $1$ & $1$ & $1$ & $n$ & $l$ & $n$ & $l$ & $1$ & $1$ & $1$ \\

  $l$ & $l$ & $l$ & $l$ & $1$ & $1$ & $1$ & $l$ & $l$ & $l$ & $l$ & $1$ & $1$ & $1$ \\

  $s$ & $s$ & $s$ & $s$ & $1$ & $1$ & $1$ & $1$ & $1$ & $1$ & $1$ & $1$ & $1$ & $1$ \\

  $d$ & $d$ & $d$ & $d$ & $1$ & $1$ & $1$ & $1$ & $1$ & $1$ & $1$ & $1$ & $1$ & $1$ \\

  $1$ & $1$ & $1$ & $1$ & $1$ & $1$ & $1$ & $1$ & $1$ & $1$ & $1$ & $1$ & $1$ & $1$ \\
\hline
\end{tabular}
\end{table}

\begin{remark}\label{}
$(1)$ In Theorem \ref{th33}$(1)$, the condition $cl(x)\vee cl(y)\in I_{e,a}$ for all $x,y\in I_{e,a} $ can not be omitted,  in general. The next example shows that the associativity of $U_{3}$ can be violated.\\
$(2)$  Similarly, in Theorem \ref{th33}$(2)$,  the condition $cl(x)\vee cl(y)\in (a,1]$ for all $x,y\in I_{e,a} $ can not be omitted.
\end{remark}

\begin{example}\label{}
Given a bounded lattice $L_{4}$ depicted in Fig.4., a uninorm $U^{*}$ on $[0,a]$ shown in Table \ref{Tab:03} and a closure operator $cl$ on $L_{4}$ defined by $cl(m)=b$ for $m\in L_{4}$ and $cl(x)=x$ for all $x\in L_{4}\setminus\{m\}$.  It is easy to see that $U^{*}\in \mathcal{U}_{\bot}^{*}$ and  $cl$ does not satisfy the condition in Theorem \ref{th33}$(1)$, i.e., $cl(m)\vee cl(k)=b\vee k=b\notin I_{e,a}$ for $m,k\in I_{e,a}$. By  Theorem \ref{th33}, we can obtain a function $U_{3}$ on $L_{4}$, shown in Table \ref{Tab:04}. Since $U_{3}(k,U_{3}(k,m))=U_{3}(k,b)=1$ and $U_{3}(U_{3}(k,k),m)=U_{3}(k,m)=b$ for  $k,m\in L_{4}$, the function $U_{3}$ does not satisfy associativity. Thus, the function $U_{3}$ is not a uninorm on $L_{4}$.
\end{example}\

\begin{minipage}{11pc}
\setlength{\unitlength}{0.75pt}\begin{picture}(600,100)
\put(270,20){\circle{2}}\put(267,12){\makebox(0,0)[l]{\footnotesize$0$}}
\put(270,40){\circle{2}}\put(273,42){\makebox(0,0)[l]{\footnotesize$e$}}
\put(270,60){\circle{2}}\put(275,60){\makebox(0,0)[l]{\footnotesize$a$}}
\put(270,80){\circle{2}}\put(275,82){\makebox(0,0)[l]{\footnotesize$b$}}
\put(270,100){\circle{2}}\put(267,108){\makebox(0,0)[l]{\footnotesize$1$}}
\put(251,61){\circle{2}}\put(237,62){\makebox(0,0)[l]{\footnotesize$m$}}
\put(251,41){\circle{2}}\put(237,42){\makebox(0,0)[l]{\footnotesize$k$}}

\put(270,21){\line(0,1){18}}
\put(270,41){\line(0,1){18}}
\put(270,61){\line(0,1){18}}
\put(251,42){\line(0,1){18}}
\put(270,21){\line(-1,1){19}}
\put(270,81){\line(0,1){18}}
\put(251,62){\line(1,1){19}}

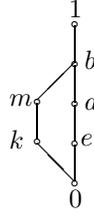
\put(200,-10){\emph{Fig.4. The lattice $L_{4}$.}}
\end{picture}
\end{minipage}

\begin{table}[htbp]
\centering
\caption{$U^{*}$ on $[0,a]$.}
\label{Tab:03}\

\begin{tabular}{c c c c c c c c}
\hline
  $U^{*}$ & $0$ & $e$ & $a$ \\
\hline
  $0$ & $0$ & $0$ & $a$ \\

  $e$ & $0$ & $e$ & $a$ \\

  $a$ & $a$ & $a$ & $a$ \\
\hline
\end{tabular}
\end{table}

\begin{table}[htbp]
\centering
\caption{$U_{3} $ on $L_{4}$.}
\label{Tab:04}\

\begin{tabular}{c c c c c c c c c}
\hline
  $U_{3} $ & $0$ & $e$ & $a$ & $k$ & $m$ & $b$ & $1$ \\
\hline
  $0$ & $0$ & $0$ & $a$ & $k$ & $m$ & $b$ & $1$ \\

  $e$ & $0$ & $e$ & $a$ & $k$ & $m$ & $b$ & $1$ \\

  $a$ & $a$ & $a$ & $a$ & $1$ & $1$ & $1$ & $1$ \\

  $k$ & $k$ & $k$ & $1$ & $k$ & $b$ & $1$ & $1$ \\

  $m$ & $m$ & $m$ & $1$ & $b$ & $b$ & $1$ & $1$ \\

  $b$ & $b$ & $b$ & $1$ & $1$ & $1$ & $1$ & $1$ \\

  $1$ & $1$ & $1$ & $1$ & $1$ & $1$ & $1$ & $1$ \\
\hline
\end{tabular}
\end{table}

\begin{remark}\label{}
Let $U_{3}$  be a uninorm in Theorem \ref{th33}.\\
$(1)$  $U_{3}$ is disjunctive, i.e., $U_{3}(0,1)=1$.\\
$(2)$ If $a=1$, then $U_{3}=U^{*}$.\\
$(3)$  $U_{3}$ is not idempotent, in general. In fact, if $(a,1)\neq \emptyset$,  then   there exists $x$ such that $x\in (a,1)$ and then $U_{3}(x,x)=1\neq x$.
This is a contradiction with that  $U_{3}$ be idempotent.\\
$(4)$  $U_{3}\in \mathcal{U}_{max}^{*}$ if and only if  $U^{*}\in \mathcal{U}_{max}^{*}$.
\end{remark}

\begin{theorem}\label{th34}
Let    $U^{*}$ be a uninorm on $[b,1]$ with a neutral element $e$  for $b\in L\setminus\{0,1\}$ and  $int$ be an interior operator.  Let $U_{4}(x,y):L^{2}\rightarrow L$ be a function defined  as follows

$U_{4}(x,y)=\begin{cases}
U^{*}(x, y) &\mbox{if } (x,y)\in [b,1]^{2},\\
x &\mbox{if } (x,y)\in (L\setminus[b,1])\times [e,1],\\
y &\mbox{if } (x,y)\in [e,1]\times (L\setminus[b,1]),\\
int(x)\wedge int(y) &\mbox{if } (x,y)\in I_{e,b}\times I_{e,b},\\
0 &\mbox{}otherwise.\\
\end{cases}$

$(1)$ Suppose that $int(x)\wedge int(y)\in I_{e,b}$ for all $x,y\in I_{e,b}$.

$\mathrm{(i)}$ If $U^{*}\in \mathcal{U}_{\top}^{*}$, then  $U_{4}$ is a uninorm with the neutral element $e\in L$ if and only if $x\parallel y$ for all $x\in I_{e,b}$ and $y\in I_{e}^{b} $.

$\mathrm{(ii)}$ If $I_{e,b}\cup I_{b}^{e}\cup (0,b)\neq \emptyset$, then $U_{4}$ is a uninorm  with the neutral element $e\in L$  if and only if  $U^{*}\in \mathcal{U}_{\top}^{*}$ and $x\parallel y$ for all $x\in I_{e,b}$ and $y\in I_{e}^{b} $.

$(2)$ Suppose that $int(x)\wedge int(y)\in [0,b)$ for all $x,y\in I_{e,b}$.

$\mathrm{(i)}$  If $x\parallel y$ for all $x\in I_{e,b}$ and $ y\in I_{e}^{b} $ and $U^{*}\in \mathcal{U}_{\top}^{*}$, then $U_{4}$ is a uninorm   with the neutral element $e\in L$.

$\mathrm{(ii)}$ If  $I_{e,b}\cup I_{b}^{e}\cup (0,b)\neq \emptyset$ and $0<int(x)\wedge int(y) $ for all $x,y\in I_{e,b}$, then $U_{4}$ is a uninorm on $L$ with the neutral element $e\in L$  if and only if $x\parallel y$ for all $x\in I_{e,b}, y\in I_{e}^{b} $ and $U^{*}\in \mathcal{U}_{\top}^{*}$.
\end{theorem}

\begin{proof}
It can be proved immediately by the proof similar to Theorem \ref{th33}.
\end{proof}

\begin{remark}
If taking $e=1$ in Theorem \ref{th34}, then $[b,e]=[b,1]$, $I_{b}^{e}=I_{b}$, $I_{e,b}\cup I_{e}^{b}\cup (e,1]=\emptyset$ and $U^{*}$ is a $t$-norm on $[b,1]$. Moreover, in this case,
 the following conditions  naturally hold.\\
$(1)$  $int(x)\wedge int(y)\in I_{e,b}$ or $int(x)\wedge int(y)\in [0,b)$ for all $x,y\in I_{e,b}$. \\
$(2)$  $U^{*}\in\mathcal{U}_{\top}^{*}$ and $x\parallel y$ for all $x\in I_{e,b}$ and $y\in I_{e}^{b}$.


By the above fact, if taking $e=1$ in Theorem \ref{th34}, then we retrieve the following $t$-norm $T_{2}^{*}: L^{2}\rightarrow L$ constructed by \c{C}ayl{\i} (\cite{GD018}, Theorem 1).

$T_{2}^{*}(x,y)=\begin{cases}
T_{e}(x, y) &\mbox{if } (x,y)\in [b,1]^{2},\\
x &\mbox{if } (x,y)\in (I_{b}\cup[0,b))\times \{1\},\\
y &\mbox{if } (x,y)\in \{1\}\times (I_{b}\cup[0,b)),\\
0 &\mbox{}otherwise.
\end{cases}$
\end{remark}

\begin{remark}
If taking $e=b$ in Theorem \ref{th34}, then $[b,1]=[e,1]$, $I_{e,b}=I_{e}$, $I_{b}^{e}\cup I_{e}^{b}\cup [b,e)=\emptyset$ and $U^{*}$ is a $t$-conorm on $[b,1]$. Moreover, in this case,
 $U^{*}\in\mathcal{U}_{\top}^{*}$
and  $x\parallel y$ for all $x\in I_{e,b}$ and $y\in I_{e}^{b}$   naturally hold.


By the above fact, if taking $e=b$ in Theorem \ref{th34}, then we retrieve the  uninorm $U_{int}: L^{2}\rightarrow L$ constructed by Zhao and Wu (\cite{BZ21}, Corollary 4.1).

$U_{int}(x,y)=\begin{cases}
S_{e}(x, y) &\mbox{if } (x,y)\in [e,1]^{2},\\
x &\mbox{if } (x,y)\in I_{e}\times [e,1]\cup[e,1]\times[0,e),\\
y &\mbox{if } (x,y)\in [e,1]\times I_{e}\cup [0,e)\times[e,1],\\
int(x)\wedge int(y) &\mbox{if } (x,y)\in I_{e}\times I_{e},\\
0 &\mbox{}otherwise.\\
\end{cases}$
\end{remark}

If we put $int(x)=x$ for $x\in L$ in Theorem \ref{th34}, then we can obtain the following result.

\begin{corollary}\label{coro341}
Let    $U^{*}$ be a uninorm on $[b,1]$ with a neutral element $e$ for  $b\in L\setminus\{0,1\}$.  Let $U_{41}(x,y):L^{2}\rightarrow L$ be a function defined  as follows

$U_{41}(x,y)=\begin{cases}
U^{*}(x, y) &\mbox{if } (x,y)\in [b,1]^{2},\\
x &\mbox{if } (x,y)\in (L\setminus[b,1])\times [e,1],\\
y &\mbox{if } (x,y)\in [e,1]\times (L\setminus[b,1]),\\
x\wedge y &\mbox{if } (x,y)\in I_{e,b}\times I_{e,b},\\
0 &\mbox{}otherwise.\\
\end{cases}$

$(1)$ Suppose that $x\wedge y\in I_{e,b}$ for all $x,y\in I_{e,b}$.

$\mathrm{(i)}$ If $U^{*}\in \mathcal{U}_{\top}^{*}$, then $U_{41}$ is a uninorm on $L$ with the neutral element $e\in L$  if and only if $x\parallel y$ for all $x\in I_{e,b}$ and $y\in I_{e}^{b} $.

$\mathrm{(ii)}$ If $I_{e,b}\cup I_{b}^{e}\cup (0,b)\neq \emptyset$, then $U_{41}$ is a uninorm on $L$ with the neutral element $e\in L$  if and only if $U^{*}\in \mathcal{U}_{\top}^{*}$ and $x\parallel y$ for all $x\in I_{e,b}$ and $y\in I_{e}^{b} $.

$(2)$ Suppose that $x\wedge y\in [0,b)$ for all $x,y\in I_{e,b}$.

$\mathrm{(i)}$  If $x\parallel y$ for all $x\in I_{e,b}, y\in I_{e}^{b} $ and $U^{*}\in \mathcal{U}_{\top}^{*}$, then $U_{41}$ is a uninorm on $L$ with the neutral element $e\in L$.

$\mathrm{(ii)}$ If  $I_{e,b}\cup I_{b}^{e}\cup (0,b)\neq \emptyset$ and $0<x\wedge y$ for all $x,y\in I_{e,b}$, then $U_{41}$ is a uninorm on $L$ with the neutral element $e\in L$ if and only if $x\parallel y$ for all $x\in I_{e,b}, y\in I_{e}^{b} $ and $U^{*}\in \mathcal{U}_{\top}^{*}$.
\end{corollary}

\begin{remark}

If taking $e=b$ in Corollary \ref{coro341}, then we retrieve the following uninorm $U_{s_{2}}: L^{2}\rightarrow L$ constructed by \c{C}ayl{\i} and Kara\c{c}al (\cite{GD17}, Theorem 3.5).

$U_{s_{2}}(x,y)=\begin{cases}
S_{e}(x, y) &\mbox{if } (x,y)\in [e,1]^{2},\\
x &\mbox{if } (x,y)\in I_{e}\times [e,1]\cup[e,1]\times[0,e),\\
y &\mbox{if } (x,y)\in [e,1]\times I_{e}\cup [0,e)\times[e,1],\\
x\wedge y &\mbox{if } (x,y)\in I_{e}\times I_{e},\\
0 &\mbox{}otherwise.\\
\end{cases}$
\end{remark}

If we put $int(x)=x\wedge b$ for $x\in L$ in Theorem \ref{th34}, then we can obtain the following corollary based on Theorem \ref{th34}(2).

\begin{corollary}\label{coro342}
Let $U^{*}$ be a uninorm on $[b,1]$ with a neutral element $e$ for $b\in L\setminus\{0,1\}$. Let $U_{42}(x,y):L^{2}\rightarrow L$ be a function defined  as follows

$U_{42}(x,y)=\begin{cases}
U^{*}(x, y) &\mbox{if } (x,y)\in [b,1]^{2},\\
x &\mbox{if } (x,y)\in (L\setminus[b,1])\times [e,1],\\
y &\mbox{if } (x,y)\in [e,1]\times (L\setminus[b,1]),\\
x\wedge y\wedge b &\mbox{if } (x,y)\in I_{e,b}\times I_{e,b},\\
0 &\mbox{}otherwise.\\
\end{cases}$

$(1)$  If $x\parallel y$ for all $x\in I_{e,b}, y\in I_{e}^{b} $ and $U^{*}\in \mathcal{U}_{\top}^{*}$, then $U_{42}$ is a uninorm  with the neutral element $e\in L$.

$(2)$ If  $I_{e,b}\cup I_{b}^{e}\cup (0,b)\neq \emptyset$ and $0<x\wedge y\wedge b$ for all $x,y\in I_{e,b}$, then $U_{42}$ is a uninorm  with the neutral element $e\in L$ if and only if $x\parallel y$ for all $x\in I_{e,b}, y\in I_{e}^{b} $ and $U^{*}\in \mathcal{U}_{\top}^{*}$.
\end{corollary}

\begin{remark}
If taking $e=b$ in Corollary \ref{coro342}, then we retrieve the following uninorm $U_{s_{3}}: L^{2}\rightarrow L$ constructed by \c{C}ayl{\i} (\cite{GD018b}, Theorem 2.23).

$U_{s_{3}}(x,y)=\begin{cases}
S_{e}(x, y) &\mbox{if } (x,y)\in [e,1]^{2},\\
x &\mbox{if } (x,y)\in (L\setminus [e,1])\times [e,1],\\
y &\mbox{if } (x,y)\in [e,1]\times (L\setminus [e,1]),\\
x\wedge y\wedge e &\mbox{if } (x,y)\in I_{e}\times I_{e},\\
0 &\mbox{}otherwise.\\
\end{cases}$
\end{remark}

\begin{remark}\label{}
$(1)$ In Theorem \ref{th34}$(1)$, we observe that the condition $int(x)\wedge int(y)\in I_{e,b}$ for all $x,y\in I_{e,b} $ can not be omitted, in general.\\
$(2)$ Similarly,  in Theorem \ref{th34}$(2)$,   the condition $int(x)\wedge int(y)\in [0,b)$ for all $x,y\in I_{e,b} $ can not be omitted, in general.
\end{remark}

\begin{remark}\label{}
Let $U_{4}$  be a uninorm in Theorem \ref{th34}.\\
$(1)$ $U_{4}$ is conjunctive, i.e., $U_{4}(0,1)=0$.\\
$(2)$ If $b=0$, then $U_{4}=U^{*}$.\\
$(3)$ $U_{4}$ is not idempotent, in general. In fact, if $(0,b)\neq \emptyset$, then there exists $x$ such that $x\in (0,b)$  and then  $U_{4}(x,x)=0\neq x$.  This is a contradiction with that  $U_{4}$ be idempotent.\\
$(4)$ $U_{4}\in \mathcal{U}_{min}^{*}$ if and only if $U^{*}\in \mathcal{U}_{min}^{*}$.
\end{remark}

\section{Conclusions}

The new construction methods for uninorms on $L$ using  a uninorm defined on a subinterval of $L$  were introduced  in  \cite{GD22} and \cite{ZX23}.   In this paper,  we continue investigating the construction methods based on different tools under some additional constraints.


  We give  some remarks about the results in this paper as follows.

(1) These methods generalize some construction  methods in the literature, such as Theorem 1 in \cite{GD16}, Theorem 1 in \cite{GD018}, Proposition 3.1 and Corollary 4.1 in \cite{BZ21},  Theorem 3.1 and Theorem 3.5 in  \cite{GD17}, Theorem 2.23 in \cite{GD018b}, and also bring some interesting results, such as Propositions \ref{co31} and \ref{co444}.

(2) Although the additional constraint conditions on the given uninorms are always needed for the construction methods, we try to investigate the additional constraints carefully and systematically and  show that some additional constraints  are necessary.

(3)  In \cite{HP21}, Zhang et al. introduced the classes of uninorms ${U}_{min}^{*}$ and $\mathcal{U}_{max}^{*}$. In this paper, we show that whether  $U$ belongs to  $\mathcal{U}_{min}^{*}$ ($\mathcal{U}_{max}^{*}$)  depends on  whether  $U^{*}$ belongs to $\mathcal{U}_{min}^{*}$ ($\mathcal{U}_{max}^{*}$),  where $U$ is constructed based on $U^{*}$.


The new  methods for uninorms on $L$ in this paper  provide  a novel perspective to study the constructions of  uninorms  and  we believe that they can work well to investigate other operators in the future research.

\end{document}